\documentclass[10pt,a4paper]{article}
\usepackage{hyperref}

\usepackage{a4wide}
\usepackage{graphicx,xcolor,textpos}
\usepackage{helvet}
\usepackage[utf8]{inputenc}
\usepackage[T1]{fontenc}
\usepackage[english]{babel}
\usepackage{amsmath}
\usepackage{amsfonts}
\usepackage{amssymb,amsthm}
\usepackage{graphicx}
\usepackage{mathtools}
\usepackage{bbold}
\usepackage{comment}
\usepackage{csquotes}
\usepackage{tikz,tikz-cd}
\usepackage{cases}
\usepackage{listings}
\usepackage{algorithm}
\usepackage{algorithmic}
\usepackage{mathdots}
\usepackage{enumerate}
\usepackage{qtree}
\newtheorem{thm}{Theorem}[section]
\newtheorem{conj}{Conjecture}
\newtheorem{lemma}[thm]{Lemma}
\newtheorem{prop}[thm]{Proposition}
\newtheorem{cor}[thm]{Corollary}
\theoremstyle{definition}
\newtheorem{df}[thm]{Definition}
\newtheorem{nota}[thm]{Notation}
\newtheorem{ques}{Question}
\theoremstyle{remark}
\newtheorem*{rem}{Remark}
\theoremstyle{definition}
\newtheorem{ex}[thm]{Example}
\theoremstyle{definition}

\newcommand{\Z}{\mathbb{Z}}
\newcommand{\N}{\mathbb{N}}
\newcommand{\Q}{\mathbb{Q}}
\newcommand{\R}{\mathbb{R}}
\newcommand{\C}{\mathbb{C}}
\newcommand{\F}{\mathbb{F}}

\newcommand{\norm}[1]{\|#1\|}
\renewcommand{\Im}{\text{Im\,}}

\DeclareMathOperator{\lcm}{lcm}
\DeclareMathOperator{\ggd}{gcd}
\DeclareMathOperator{\GL}{GL}
\DeclareMathOperator{\RF}{RF}

\DeclareMathOperator{\Tr}{Tr}
\DeclareMathOperator{\ad}{ad}
\DeclareMathOperator{\rank}{rank}
\DeclareMathOperator{\Res}{Res}

\title{Residual Finiteness Growth in Two-Step Nilpotent Groups}
\author{Jonas Der\'e and Joren Matthys\thanks{Corresponding author: \href{jonas.dere@kuleuven.be}{Jonas Der\'e}. KU Leuven Campus Kulak Kortrijk, Department of Mathematics, Research unit `Algebraic Topology and Group Theory', B-8560 Kortrijk, Belgium. The authors were supported by Internal Funds KU Leuven (project number 3E220559).}}
\date{\vspace{-1cm}}
\begin{document}
	\maketitle
	\begin{abstract}
	Given a finitely generated residually finite group $G$, the residual finiteness growth $\RF_G: \N \to \N$ bounds the size of a finite group $Q$ needed to detect an element of norm at most $r$. More specifically, if $g\in G$ is a non-trivial element with $\norm{g}_G \leq r$, so $g$ can be written as a product of at most $r$ generators or their inverses, then we can find a homomorphism $\phi: G \to Q$ with $\phi(g) \neq e_Q$ and $|Q| \leq \RF_G(r)$. The residual finiteness growth is defined as the smallest function with this property. This function has been bounded from above and below for several classes of groups, including virtually abelian, nilpotent, linear and free groups. 
	
	However, for many of these groups, the exact asymptotics of $\RF_G$ are unknown (in particular this is the case for a general nilpotent group), nor whether it is a quasi-isometric invariant for certain classes of groups. In this paper, we make a first step in giving an affirmative answer to the latter question for $2$-step nilpotent groups, by improving the polylogarithmic upper bound known in literature, and to show that it only depends on the complex Mal'cev completion of the group. If the commutator subgroup is one- or two-dimensional, we prove that our bound is in fact exact, and we conjecture that this holds in general. 
	\end{abstract}
	\section{Introduction}
	A group $G$ is called \emph{residual finite} if for all non-trivial elements $g\in G$ there exists a homomorphism to a finite group $\phi: G \to Q$ such that $\phi(g) \neq e$. Various classes of groups are known to be residually finite, for example finitely generated nilpotent groups, polycyclic groups, finitely generated linear groups, fundamental groups of compact $3$-manifolds, \ldots Interestingly, it is an open question whether hyperbolic groups are residually finite, we refer to the introduction of \cite{tholozan2022residually} for some history on this problem.
	
	In \cite{bou2010quantifying}, Bou-Rabee began to quantify residual finiteness. Given a finitely generated group $G$, the author defined the (normal) \emph{residual finiteness growth} $\RF_G$ to be the minimal function such that if an element $g$ has word norm $0 \neq \norm{g}_G \leq r$, then $\phi$ can be found as above with $|Q| \leq \RF_G(r)$. Here, $\norm{\cdot}_G$ is the word norm with respect to some finite generating set $S$. If $\RF_G$ is considered up to the equivalence relation introduced in Section \ref{sec_growth}, then the function $\RF_G$ is independent of the choice of $S$. This way, the residual finiteness growth becomes a group invariant.
	
	Since the birth of this research topic, bounds on $\RF_G$ have been established for several classes of groups, including linear groups \cite{franz2017quantifying}, lamplighter groups \cite{bou2019residual}, free groups \cite{bradford2019short}, \ldots \hspace{1mm} Exact asymptotics are known for only few groups and classes, for example the first Grigorchuk group \cite{bou2010quantifying}, virtually abelian groups \cite{dere2023residual} and $S$-arithmetic subgroups of higher rank Chevalley groups \cite{bou2012quantifying}. We refer to the survey article \cite{survey2022} for more details.
	
Many open questions remain in this relatively recent topic. The exact connection between the residual finiteness growth of groups and group constructions such as forming semi-direct or wreath products is unknown. Also, as exact results on $\RF_G$ are lacking for many classes, it is unclear whether some classes can be characterized via their residual finiteness growth. A positive result in this direction has been shown in \cite{bou2011asymptotic}, where the authors proved that virtually nilpotent groups within the class of linear groups are those for which $\RF_G$ is bounded by $\log^k$ for some $k\geq 0$. However, one may ask whether the result can be strengthened to the following:
	\begin{conj}
		A finitely generated (residually finite) group $G$ is virtually nilpotent if and only if $\RF_G$ equals $\log^k$ for some $k\geq 0$.
	\end{conj}
\noindent Both implications of the conjecture are still open. In \cite{dere2023residual}, it was shown that at least virtually abelian groups have residual finiteness growth equal to $\log^k$, therefore partially confirming one direction of this conjecture.
	
	In this article, we will focus on nilpotent groups. From \cite{bou2010quantifying}, we know that their residual finiteness growth is bounded above by $\log^{h(G)}$, where $h(G)$ is the Hirsch length of the nilpotent group. Later, an upper bound of the form $\log^{\psi(G)}$ with $\psi(G)$ possibly much smaller than $h(G)$ was found in \cite{pengitore2018corrigendum}. This bound was known to be non-exact and to only depend on the so-called rational Mal'cev completion $G^\Q$ of a torsion-free group $G$. This raises the following question:
	\begin{ques}
		Let $G$ be a torsion-free nilpotent group. Does $\RF_G$ equal $\log^k$, where $k$ depends on the Mal'cev completion $G^\F$ over some field $\F$?
	\end{ques}
	
	Focusing on the two-step nilpotent case, we will establish a new upper bound, which improves the previous one.
	\begin{thm} \label{thm_main_upper_bound}
		Let $G$ be a finitely generated, infinite, two-step nilpotent group, then $\RF_G$ is smaller than $\log^{2k(G)+1}$ for some $k(G)\in \N$.
	\end{thm}
\noindent	The constant $k(G)$ only depends on the complex Mal'cev completion $G^\C$ and is better that the previous known bound, as seen from the following proposition:
	\begin{prop} \label{prop_psiG_dG}
		There exists a family $(G_n)_{n \in \N}$ of torsion-free finitely generated two-step nilpotent groups $G_n$, such that $k(G_n) = 1$ for all $n \in \N$, implying that $\RF_G = \log^3$, but such that $\psi(G_n) \to \infty$.
	\end{prop}
	We also prove a lower bound which agrees with the upper bound under a certain assumption, see Theorem \ref{thm_lower_bound}. If the commutator subgroup of the group has rank at most $2$, then we show that this condition is always satisfied. In fact, the authors have no knowledge of any example where the condition for the lower bound is not satisfied, providing evidence that the exact asymptotics of $\RF_G$ have been found for two-step nilpotent groups.
	
 	The outline of this paper is as follows. In Section \ref{sec:prel}, we formally (re)introduce the notion of residual finiteness growth and afterwards we recall the notion of Mal'cev completions. We also introduce an alternative description of two-step nilpotent groups in Section \ref{sec_twostep}. In Section \ref{sec_upper}, we continue to prove Theorem \ref{thm_main_upper_bound}, which is in fact a rephrased version of Theorem \ref{thm_upper_bound}. Proposition \ref{prop_psiG_dG} will be established in Section \ref{previousbound}. The subclass of the result refers to a class of non-singular groups. The notion of non-singularity is also defined in this section. Next, in Section \ref{sec_lower}, we establish the lower bound and its connection to the upper bound. Finally, in Section \ref{sec_examples}, we show the condition for the lower bound is satisfied, giving exact asymptotics, at least when the commutator subgroup has rank at most $2$.
	\section{Preliminaries}\label{sec:prel}
	\subsection{Residual Finiteness Growth} \label{sec_growth}
	In this section, we will introduce the notion of residual finiteness growth, as it was originally coined in \cite{bou2010quantifying}. It should be noted that in more recent papers, this function is sometimes called `normal residual finiteness growth', where `residual finiteness growth' is then used for a different non-equivalent function, as for example in \cite{bou2015residual}. From now on, the set $\{1,2, \ldots\}$ of natural numbers is denoted by $\N$, and we write $\N^\ast$ for $\N\cup \{0\}$.
	
	Throughout this section, let $G$ be a finitely generated, residually finite group with finite generating set $S$. The neutral element in a group will be denoted by the letter $e$. Let $e\neq g\in G$ be some non-trivial element. As $G$ is residually finite, there exists by definition a homomorphism $\phi: G \to Q$ to a finite group $Q$ such that $\phi(g)\neq e$. The divisibility function associates the minimal size of $Q$ to $g$:
		\begin{df}
		The \emph{divisibility function} $D_G: G \to \mathbb{N}\cup\{\infty\}$ is defined as
		\begin{equation*}
			D_{G}(g) =  \begin{cases} \min\left\{[G:N]\mid g\notin N, N \lhd G
				
				\right\} & \text{for } g \neq e,\\ \infty & \text{for } g = e. \end{cases}
		\end{equation*}
	\end{df}
	
\noindent	Since $G$ is finitely generated by $S$, we may define a word norm on $G$ as follows:
	\begin{df}
		The \emph{word norm} on $G$ via $S$ is defined as
		$$\norm{g}_{S} = \min\{k\mid g = s_1^{\pm1} \ldots s_k^{\pm1}, s_i \in S, k \in \mathbb{N}^\ast \}.$$
	\end{df}
\noindent	A metric ball $B_S(r)$ centered around $e$ with radius $r\in\mathbb{R}^+$ is therefore given by
	\begin{align*}
		B_S(r) &= \left\{g\in G \mid \norm{g}_S \leq r	\right\} \\ &= \left\{s_1^{\pm1} \ldots s_k^{\pm1} \mid s_i \in S, k \leq r	\right\} .
	\end{align*}
The residual finiteness growth encapsulates the interplay between the metric given by $\norm{\cdot}_S$ and the divisibility function $D_G(\cdot)$.
		\begin{df}
		The \emph{residual finiteness growth} with respect to $B_S$ is given by the function
		$$\RF_{G,S}: \mathbb{R}_{\geq 1} \to \mathbb{N}\cup\{\infty\} : r \mapsto \max\{D_G(g) \mid e\neq g \in B_S(r)\}.$$
	\end{df}
\noindent	It is the minimal function such that if $e\neq g\in B_S(r)$ is arbitrary, there exists a homomorphism $\varphi: G\to Q$ to a finite group $Q$ with $|Q| \leq \RF_{G,S}(r)$ such that $\varphi(g) \neq e$.
	
	Suppose now that $\left(B\left(n\right)\right)_{n\in\mathbb{N}}$ is some sequence of non-decreasing subsets of $G$ for which there exists a $C>1$ such that
	\begin{equation} \label{eq_metric}
	B_S\left(\frac{1}{C}n\right) \subset B(n) \subset B_S(Cn) 
	\end{equation}
	for all $n$ sufficiently large. In that case,
	$$\RF_{G,S}\left(\frac{1 }{C}n\right) \leq  \max\{D_G(g) \mid e\neq g \in B(n)\} \leq \RF_{G,S}(Cn).$$
	Let us therefore define the following relations:
	\begin{df}
	Let $f,g: \mathbb{R}_{\geq 1} \to \mathbb{R}_{\geq 1}$ be non-decreasing functions. We write 
	\begin{align*}f &\preceq g \Leftrightarrow \exists C >0: \forall r \geq \max\{1,1/C\}: f(r) \leq Cg(Cr);\\
		f&\approx g \Leftrightarrow f\preceq g \text{ and } g \preceq f\end{align*}
	\end{df}
	We have now made the observation below:
	\begin{lemma}
		If $\left(B(n)\right)_{n\in\mathbb{N}}$ is some sequence of non-decreasing subsets of $G$ satisfying equation \eqref{eq_metric}, then $$\RF_{G,S}(r) \approx \max\{D_G(g) \mid e\neq g \in B(\lfloor r \rfloor)\}.$$
	\end{lemma}
	In particular, this condition is satisfied if $B(n) = B_T(n)$ for some alternative finite generating set $T$ of $G$, implying that $\RF_{G,S} \approx \RF_{G,T}$ for two different finite generating sets. Up to the equivalence $\approx$, the residual finiteness growth becomes a group invariant. We denote the equivalence class of $\RF_{G,S}$ by $\RF_G$. The lemma also shows that we may replace $B_S(r)$ in the definition of $\RF_{G,S}$ by balls that non-necessarily come from a word metric. For example, on $\Z^m$, we may replace the word norm by the euclidean norm.

	\subsection{Mal'cev Completions}
	In this section, we will recall the notion of Mal'cev completions of finitely generated, torsion-free nilpotent groups. Let $G_1$ and $G_2$ be two finitely generated torsion-free nilpotent groups.
	
	\begin{df}
		A group $G$ is called \emph{nilpotent} if there exists a finite-length series of normal subgroups of $G$ of the form
		$$\{e\} = G_0 \lhd G_1 \lhd \dots \lhd G_{c-1} \lhd G_c = G$$
		with $G_{i+1}/G_i \leq Z(G/G_i)$, where $Z(H)$ denotes the center of $H$. The minimal $c$ for which such a series exists is called the step length.
	\end{df}
	\begin{df}
	A torsion-free, finitely generated nilpotent group is called an \emph{$\mathcal{I}$-group}.
	\end{df}
	
	It is known that $\mathcal{I}$-groups are up to isomorphisms exactly the subgroups of $\Tr_1(n, \Z) \leq \GL(n, \Z)$, i.e. the $n\times n$ upper-triangular matrices with ones on the diagonal, see \cite[Chapter 5, Theorem 2]{sega83-1}. Consider also $\Tr_0(n, \C)$ of strictly upper triangular $n\times n$ matrices over $\C$. The exponential map of matrices forms a bijection $\exp: \Tr_0(n, \C) \to \Tr_1(n, \C)$, see \cite[Theorem 1.2.1]{corwin1990representations}, with its inverse given by $\log: \Tr_1(n, \C) \to \Tr_0(n, \C)$. Using this bijection, one can define a group isomorphism
	$$\Phi: (G, \cdot) \to (\log(G), \ast): g \in \Tr_1(n, \Z) \mapsto \log(g) \in \Tr_0(n,\C)$$
	where $v\ast w$ with $v,w\in \log(G)$ is given by the Campbell-Baker-Hausdorff formula
	$$v\ast w := v+w + \frac{1}{2}[v,w] + \sum_{e=3}^\infty q_e(v,w),$$
	with $[v,w] = vw-wv$ the Lie-bracket of matrices and $q_e(v,w)$ a rational linear combination of nested Lie bracket of lenght $e$, see \cite[Section 1.2]{corwin1990representations} for a more detailed description.
	
	Let $\F \subset \C$ be a field. Consider $\F\log(G)$, i.e. the $\F$-span of the matrices in $\log(G)$. Using that $v\ast w \in \log(G)$ for all $v,w \in \log(G)$ and the Campbell-Baker-Hausdorff formula is defined via Lie-brackets with rational coefficients, we have that $v'\ast w' \in \F\log(G)$ if $v',w' \in \F\log(G)$ (see \cite{sega83-1}), implying that $(\F\log(G), \ast)$ forms a group. One can think of it as the Mal'cev completion of $G$ over the field $\F$, or simply its $\F$-completion, although the following definition is more conventional:
	\begin{df}
		Let $\F$ be a field. The \emph{Mal'cev completion} of an $\mathcal{I}$-group $G$ over the field $\F$, or simply its $\F$-completion, is given by 
		$$G^\F := \exp( \F \log(G)).$$
		Here, $\exp$ is the matrix exponential and $\log$ its inverse on $\Tr_1(n, \C)$. The set $\F \log(G)$ is the $\F$-span of the matrices in $\log(G)$.
	\end{df}
	\begin{rem}
		The exponential map maps the Lie algebra $\Tr_0(n, \C)$ of strictly upper-triangular matrices diffeomorphically to the Lie group $\Tr_1(n, \C)$. The $\F$-completion is the group corresponding to the smallest Lie algebra over $\F$ containing $\log(G)$.
	\end{rem}
	\begin{rem}
		It should be noted that the embedding of the $\mathcal{I}$-group $G$ in $\Tr_1(n, \Z)$ is not unique. However, different embeddings result in isomorphic Mal'cev completions. Therefore, we will assume in the remainder of the article that the embedding is fixed.
	\end{rem}
	It is generally conjectured (see for example \cite[Conjecture 19.114]{MR3793294}) that two $\mathcal{I}$-groups are quasi-isometric if and only if their real Mal'cev completions are isomorphic. Here, the `if' statement is a consequence of the Schwarz-Milnor lemma and the fact that both groups form a cocompact lattice in the same Lie group. In this paper, we relate $\RF_G$ of a two-step $\mathcal{I}$-group $G$ to its complex Mal'cev completion. In light of this conjecture, this provides some evidence that $\RF_G$ is a quasi-isometric invariant, at least in all known examples.
		\subsection{Description of two-step Nilpotent Groups} \label{sec_twostep}
	In this part, we rewrite the description of a general two-step $\mathcal{I}$-groups for our purposes. This is based on the description of two-step $\mathcal{I}$-groups using alternating bilinear maps $\Z^m \times \Z^m \to \Z^n$ from \cite[Chapter 11C, Proposition 4]{sega83-1}. More precisely, we have the following result:
	\begin{df}
		A map $\varphi: \Z^m \times \Z^m \to \Z^n$ is called
		\begin{enumerate}[(i)]
			\item \emph{bilinear}, if $\varphi(w_1+kw_2, w_3) = \varphi(w_1,w_3) + k\varphi(w_2,w_3)$ and $\varphi(w_1,w_2 + k w_3) = \varphi(w_1,w_2) + k\varphi(w_1,w_3)$ for all $w_1,w_2,w_3 \in \Z^m$ and all $k \in \Z$,
			\item \emph{alternating}, if $\varphi(w_1,w_2) = - \varphi(w_2, w_1)$ for all $w_1,w_2 \in \Z^m$,
			\item \emph{full}, if $\langle \varphi(w_1,w_2) \mid w_1,w_2 \in \Z^m \rangle$ is a finite index subgroup of $\Z^n$.
		\end{enumerate}
	\end{df}
	\begin{df}
		The free two-step nilpotent group on $m$ generators $F_{m,2}$ is defined as the quotient group $F_m/\gamma_3(F_m)$, where $F_m$ is the free group on $m$ generators and $\gamma_3(F_m)$ is the third term of its lower central series.
	\end{df}
	Note that $(F_{m,2}/[F_{m,2},F_{m,2}]) \cong \Z^m$. If $x$ is an element of $F_{m,2}$, we will write $\overline{x} \in \Z^m$ for the element $x[F_{m,2},F_{m,2}]$ in $(F_{m,2}/[F_{m,2},F_{m,2}]) \cong \Z^m$.
	\begin{thm}
		Each two-step $\mathcal{I}$-group is isomorphic to a group of the form $H_\varphi = (F_{m,2} \times \Z^n)/K_\varphi$ with $K_\varphi = \langle ([x,y]^{-1}, \varphi(\overline{x}, \overline{y}))\mid x,y \in F_{m,2}\rangle$, where $F_{m,2}$ is the free two-step nilpotent group of rank $m$ and $\varphi: (F_{m,2}/[F_{m,2},F_{m,2}])\times (F_{m,2}/[F_{m,2},F_{m,2}]) \to \Z^n$ is a full, alternating, bilinear map. 
	\end{thm}
	There is also a classification of these groups up to isomorphism, but this is not important for our discussion. In light of this characterization, we have the following definition:
	\begin{df}
		We say $H_\varphi$ is an \emph{$\mathcal{I}(m,n)$-group} if $\varphi$ is a full, alternating, bilinear map $\Z^m\times\Z^m \to \Z^n$.
	\end{df}
	\begin{rem}
		Since $\varphi$ has to be full, one can check that $m \leq n(n-1)/2$. Conversely, if $m \leq n(n-1)/2$, the class of $\mathcal{I}(m,n)$-groups is non-empty.
	\end{rem}
	
	In the discussion that follows, we will show that every group $H_\varphi$ of the theorem above can be written in a different form $G_\varphi$, which is more convenient for our computations. We will do so by `replacing' every element $(x,v)K_\varphi$ by a well-chosen representative in this coset of $F_{m,2}\times \Z^n$. 
	\begin{df}
		Let $\varphi: \Z^m\times \Z^m \to \Z^n$ be a full, alternating, bilinear form. With respect to the standard bases on $\Z^m$ and $\Z^n$, $\varphi$ is represented by $n$ skew-symmetric $m\times m$-matrices $A_1$ up to $A_n$, i.e.
		$$\varphi(w_1,w_2) = (w_1^TA_1w_2, \dots , w_1^TA_n w_2)^T \in \Z^n.$$
		Let $A_i^L$ be the matrix defined as the matrix $A_i$, where the upper triangular part is set to zero.
		
		Define the group $(G_{\varphi}, \cdot)$ as follows:
		\begin{itemize}
			\item as a set, $G_\varphi$ equals $\Z^m\times \Z^n$,
			\item the operation is given by $$(w_1,v_1)\cdot (w_2,v_2) = (w_1+w_2,v_1+v_2+\varphi^L(w_1,w_2)) ,$$
			where $\varphi^L$ is the bilinear map defined by the matrices $A_i^L$ ($1\leq i \leq n$).
		\end{itemize}
	It will be clear that $G_\varphi$ satisfies the group axioms from the next proposition.
	\end{df}

	\begin{prop} \label{prop_iso_Gvarphi}
		The groups $H_\varphi$ and $G_\varphi$ are isomorphic.
	\end{prop}
	We start by observing what happens to the coset $(x,v)K_\varphi$ of $H_\varphi$ if the order of generators in $x$ is changed.
	\begin{lemma}
		Let $(x,v)K_\varphi$ be an element of $H_\varphi$ with $\varphi: \Z^m\times \Z^m \to \Z^n$. If $x\in F_{m,2}$ is given by the word $AabB$ where $a$ and $b$ are non-empty sub-words, then
		$$(x,v)K_\varphi = (AabB,v)K_\varphi = (AbaB,v+\varphi(\bar{a}, \bar{b}))K_\varphi.$$
	\end{lemma}
	\begin{proof}
		We claim that $(ab,0)K_\varphi = (ba, \varphi(\bar{a}, \bar{b}))K_\varphi$, i.e. $(ab, 0)\cdot (ba,\varphi(\overline{a}, \overline{b}))^{-1} \in K_\varphi$.
		Indeed, we can rewrite this product in the following way:
		\begin{equation*}
			\begin{split}
				(ab, 0)\cdot (ba,\varphi(\overline{a}, \overline{b}))^{-1} & = (ab, 0)\cdot (a^{-1}b^{-1},-\varphi(\overline{a}, \overline{b}))\\
				& = ([a^{-1},b^{-1}],\varphi(\overline{a^{-1}}, \overline{b}))\\
				& =  ([a^{-1},b]^{-1},\varphi(\overline{a^{-1}}, \overline{b})).\\
			\end{split}
		\end{equation*}
		In the last step, we used that $F_{m,2}$ is two-step nilpotent, and therefore, the commutator is bilinear. By the definition of $K_\varphi$, the claim is proven. 
		
		Now, the normality of $K_\varphi$ and the claim imply that
		$$(A,0)(ab,0)(B,v)K_\varphi = (A,0)(ba, \varphi(\bar{a}, \bar{b}))(B,v)K_\varphi .$$
		This gives the statement of the lemma.
	\end{proof}

	We now proceed to show the isomorphism.
	\begin{proof}[Proof of Proposition \ref{prop_iso_Gvarphi}]
		Fix standard generators $e_1, \ldots, e_m$ for $F_{m,2}$, which have now a fixed order. Their quotients in $\Z^m \cong F_{m,2}/[F_{m,2}, F_{m,2}]$ form a $\Z$-basis, denoted by $\overline{e_1}$ up to $\overline{e_m}$.
		
		Consider an arbitrary element $(x,v)K_\varphi$ in $H_\varphi$. This coset of $K_\varphi$ in $H_\varphi$ has a representative of the form
		$$(e_1^{k_1}e_2^{k_2}\ldots e_m^{k_m},v') = (\prod_{i=1}^me_i^{k_i},v') \text{ with } k_i \in \Z, v'\in \Z^n.$$
		Indeed, the previous lemma allows us to change the order of the generators $e_1$ to $e_m$ in the first component by changing the second component. We will call this form the base form, and we will show in Lemma \ref{lem_unique_base_form} that it is unique.
		
		We proceed to show how the base form behaves with respect to products. Let $(\prod_{i=1}^me_i^{k_i},v_1)$ and $(\prod_{i=1}^me_i^{l_i},v_2)$ denote two base forms. For ease of notation, we write $w_1$ for the vector $(k_1, \dots , k_m)^T \in \Z^m$ and analogously $w_2$ for the other.
		
		We have 
		\begin{equation} \label{eq_rewr_base_forms}
		(\prod_{i=1}^me_i^{k_i},v_1)\cdot (\prod_{i=1}^me_i^{l_i},v_2) = (\prod_{i=1}^me_i^{k_i}\cdot \prod_{i=1}^me_i^{l_i},v_1 + v_2).
		\end{equation}
		We need to bring this to base form. However, using that $(AabB,v)$ and $(AbaB,v+\varphi(\overline{a}, \overline{b}))$ represent to same cosets, one can show via induction that Equation \eqref{eq_rewr_base_forms} equals
		$$ (\prod_{i=1}^me_i^{k_i+l_i},v_1 + v_2+ \sum_{i>j} k_il_j \varphi(\overline{e_i}, \overline{e_j})) = (\prod_{i=1}^me_i^{k_i+l_i},v_1 + v_2+ \varphi^L(w_1,w_2)).$$
		
		By the observation, a homomorphism between $G_\varphi$ and $H_\varphi$ is now given by
		$$G_\varphi \to H_\varphi: (w,v) \mapsto (\prod_{i=1}^me_i^{k_i},v) K_\varphi \text{ with } w = (k_1, \dots , k_m)^T.$$
		It is surjective by the existence of the base form. For injectivity, we need to show that $(\prod_{i=1}^me_i^{k_i},v) \in K_\varphi$ implies that $k_i = 0, v = 0$.
		This corresponds to the unicity of the base form. It is proven in Lemma \ref{lem_unique_base_form} below.
	\end{proof}
	\begin{lemma} \label{lem_unique_base_form}
		Let $H_\varphi$ be the two-step $\mathcal{I}$-nilpotent group $(F_{m,2}\times \Z^n)/K_\varphi$. Then $(\prod_{i=1}^me_i^{k_i},v) \in K_\varphi$ if and only if $v=0$ and $k_i = 0$ for all $1\leq i \leq m$.
	\end{lemma}
	\begin{proof}
		The `if' statement is clear. For the other direction, we know by \cite[Chapter 11C, Exercise 7]{sega83-1} that $\Z^n \cong (0\times\Z^n)K_\varphi/K_\varphi$ is the isolator of $[H_\varphi,H_\varphi]$ and that $H_\varphi/\Z^n \cong \dfrac{F_{m,2}}{[F_{m,2},F_{m,2}]} \cong \Z^m$. Hence, the group $H_\varphi$ fits in the short exact sequence
		\[\begin{tikzcd}
			1 & {\Z^n} & H_\varphi & {\dfrac{F_{m,2}}{[F_{m,2},F_{m,2}]} \cong \Z^m} & 1
			\arrow[from=1-1, to=1-2]
			\arrow["i", hook, from=1-2, to=1-3]
			\arrow["\pi", two heads, from=1-3, to=1-4]
			\arrow[from=1-4, to=1-5]
		\end{tikzcd}\]
		with $i(v) = (e,v)K_\varphi$ and $\pi((x,v)K_\varphi) = \overline{x}$. In particular, since $(\prod_{i=1}^me_i^{k_i},v)K_\varphi = K_\varphi$ by assumption, we have that
		$$0 = \pi(K_\varphi) = \pi((\prod_{i=1}^me_i^{k_i},v)K_\varphi) = \overline{\prod_{i=1}^me_i^{k_i}} = (k_1, \dots , k_m),$$
		and
		$$0 = i^{-1}(K_\varphi) = i^{-1}((e,v)K_\varphi) = v.$$
	\end{proof}
		\begin{rem}
		The inverse of $(w,v)$ in $G_\varphi$ is given by $(-w,-v+\varphi^L(w,w))$. Therefore, we observe that $[(w_1,v_1),(w_2,v_2)] = (0,\varphi^L(w_1,w_2)-\varphi^L(w_2,w_1))$. This equals precisely $(0, \varphi(w_1,w_2))$. This shows how to recover the map $\varphi$ from the group $G_\varphi$.
	\end{rem}
	
	We end this section by providing an example that we will use as recurrent example throughout the article. We also make a note concerning the Mal'cev completions of the groups $G_\varphi$.
	\begin{ex} \label{ex_Heis}
		The Heisenberg group over $\Z[i]$ is given by the matrix group
		$$H_3(\Z[i]) = \left\{ \begin{pmatrix}1&a_1+b_1i & a_2 + b_2i\\0&1&a_3+b_3i\\0&0&1\end{pmatrix}\mid a_i, b_i\in \Z\right\}.$$
		This matrix group can also be represented by $G_\varphi$ with $\Z^4\times \Z^2$ as a set and with $\varphi$ given by the matrices 
$$\begin{pmatrix}0&0&-1&0\\0&0&0&1\\1&0&0&0\\0&-1&0&0\end{pmatrix},\begin{pmatrix}0&0&0&-1\\0&0&-1&0\\0&1&0&0\\1&0&0&0\end{pmatrix},$$
		using the isomorphism
		$$\begin{pmatrix}1&a_1+b_1i & a_2 + b_2i\\0&1&a_3+b_3i\\0&0&1\end{pmatrix} \mapsto (a_3,b_3,a_1,b_1,a_2,b_2) .$$ 
		In \cite[Proposition 2.2]{pengitore2018corrigendum}, it is shown that $\RF_{H_3(\Z[i])} \preceq \log^4$. We will show in this paper that the residual finiteness growth of this $\mathcal{I}(4,2)$-group equals $\log^3$.
	\end{ex}
	\begin{rem}
		The groups $G_\varphi$ are as sets given by $\Z^m\times \Z^n$ with operation defined by $\varphi^L$. In Lemma \ref{lem_Lie_alg_twostep} and its remark, the $\F$-completion of $G_\varphi$, i.e. $(G_\varphi)^\F$, will be shown to be given by the set $\F^m\times \F^n$. Its operation is given by
		$$(w_1,v_1)\cdot (w_2,v_2) = (w_1+w_2,v_1+v_2+(\varphi^L)^\F(w_1,w_2)) ,$$
		where $(\varphi^L)^\F$ is the extension of $\varphi^L$ to a bilinear map $\F^m\times \F^m \to \F^n$. In this article, we will be interested in the complex completion $(G_\varphi)^\C$. 
		
		Throughout this article, let $\Z_p$ denote the field of $p$ elements for a prime $p$. In a way specified later, the behavior of the complex completion will be linked with the behavior of $(G_\varphi)^{\Z_p}$ for well-chosen primes $p$. This is the finite group with $\Z_p^m\times \Z_p^n$ as a set. The map $(\varphi^L)^{\Z_p}$ defining the operation is then given by the map 
		$$(\varphi^L)^{\Z_p}: \Z_p^m\times \Z_p^m \to \Z_p^n: (w_1,w_2) \mapsto (w_1^TA_{1,p}^Lw_2, \ldots,w_1^TA^L_{n,p} w_2),$$
		where $A_{i,p}^L$ is the matrix obtained by reducing the entries of $A_i^L$ modulo $p$.
		We will call $(G_\varphi)^{\Z_p}$ the $\Z_p$-reduction of $G$.
	\end{rem}

	\section{Proof of the Upper Bound} \label{sec_upper}
	In this section, we provide an upper bound for the residual finiteness growth in two-step nilpotent groups. This bound will depend on the complex Mal'cev completion. In the next section, we will also show that this bound improves the currently best known bound in literature (see \cite[Theorem 3.1]{pengitore2018corrigendum}). From now on, $G_\varphi$ will always be identified with $\Z^m\times \Z^n$ as a set.
	\subsection{Normal Subgroups}
	We begin by analyzing normal subgroups in $G_\varphi$. This will allow us to simplify the divisibility function in this setting, leading to the upper bound.
	\begin{rem}
		Consider the short exact sequence
		\[\begin{tikzcd}
			1 & {\Z^n} & G_\varphi & {\Z^m} & 1,
			\arrow[from=1-1, to=1-2]
			\arrow["i", hook, from=1-2, to=1-3]
			\arrow["\tau", two heads, from=1-3, to=1-4]
			\arrow[from=1-4, to=1-5]
		\end{tikzcd}\]
		where $i: \Z^n \to G_\varphi: v \mapsto (0,v)$ and $\tau: G_\varphi \to \Z^m: (w,v) \mapsto w$. Let $N$ be a normal subgroup of $G_\varphi$. Due to the sequence above, we have the following equality:
		$$[G:N] = [\Z^n: N\cap \Z^n]\cdot [\Z^m: \tau(N)],$$
		where $N\cap \Z^n = \{v\in \Z^n\mid (0,v) \in N\}$ and $\tau(N) = \{w \in \Z^m\mid \exists v \in \Z^n: (w,v)\in N\}$. Both $\tau(N)$ and $N\cap \Z^n$ are (normal) subgroups of free abelian groups. Hence, they may be interpreted as $\Z$-modules $B$ and $D$ respectively. 
	\end{rem}
	\begin{lemma} \label{lem_existence_NBD}
		Let $B$ be a subgroup of $\Z^m$ with basis $\{w_1,w_2 \dots w_m\}$ and $D$ a subgroup of $\Z^n$. Suppose $\varphi(B, \Z^m) \subset D$, then
		$$N_{B,D} = \langle (w_i,0), (0,v)\mid 1\leq i \leq m, v\in D\rangle $$
		is a normal subgroup of $G_\varphi$ of index $[\Z^m:B] \cdot [\Z^n:D]$.
	\end{lemma}
	\begin{proof}
		Let us first show that $N_{B,D} \cap \Z^n \subset D$. Indeed, take any $(0,v_0) \in N_{B,D}$ and write $(0,v_0)$ as a product of generators $g_1^{\epsilon_1}g_2^{\epsilon_2}\dots g_l^{\epsilon_l}$ with $l \in \mathbb{N}$ and $\epsilon_i \in \{1,-1\}$. If $g_i$ is a generator of the form $(0, v)$ for some $v \in D$, then it commutes with all other generators. Therefore, we may rewrite $(0,v_0)$ to 
		$$(0,v_0) = \left(\prod_{j=1}^k (w_{i_j},0)^{\epsilon_j}\right) \cdot \left( \prod_{j=k+1}^l (0, v_j)^ {\epsilon_j}\right),$$
		with $1 \leq i_j \leq m$ and $v_j \in D$. Since $\prod_{j=k+1}^l (0, v_j)^{\epsilon_j} $ is of the form $(0, v)$ with $v \in D$, it suffices to show that if
		$$(0,v_0) = \prod_{j=1}^k (w_{i_j},0)^{\epsilon_j},$$
		then $v_0 \in D$. 
		Recall that $(w, v)(w', v') = (w+w', v + v' + \varphi^L(w,w'))$ and $(w,v)^{-1} = (-w,-v+ \varphi^L(w,w))$. In particular, we have linear behavior in the first component. Looking at the projection onto this component of the product above, we get an expression
		$$\sum_{i=1}^m(n_i - n_i')w_i = 0,$$
		where $n_i$ denotes the times $(w_i,0)$ appears in the product and $n'_i$ denotes the times its inverse appears.	Since $\{w_1, \dots w_m\}$ is a basis of $B$, we know $n_i = n_i'$. Hence, we may form pairs of the form $\{(w_i,0) ; (-w_i, \varphi^L(w_i,w_i))\}$ of mutual inverses. In particular, $k$ is even, say $k = 2d$.
		
		We proceed by induction on the number $d$ to prove the claim. If $d = 1$, we have two possible options: either $(w_i,0)(-w_i, \varphi^L(w_i,w_i))$ or $(-w_i, \varphi^L(w_i,w_i))(w_i,0)$. Both equal $(0,0)$, satisfying the result.
		
		Suppose the claim holds up to value $d$, then we prove it for the value $d+1$. We assume the last element of the product is $(w_i,0)$. The other case is analogous. Suppose it is paired up with $(-w_i, \varphi^L(w_i,w_i))$ at position $1\leq e \leq 2d+1$. Let $(w_{<e},v_{<e})$ and $(w_{>e},v_{>e})$ denote the elements $\prod_{j=1}^{e-1} (w_{i_j},0)^{\epsilon_j}$ and $\prod_{j=e+1}^{2d+1} (w_{i_j},0)^{\epsilon_j}$ respectively.
		By the induction hypothesis, we know that
		\begin{equation*}
				(w_{<e},v_{<e})\cdot(w_{>e},v_{>e}) = (0, v_{<e}+v_{>e}+\varphi^L(w_{<e},w_{>e}))\in (0, D).
		\end{equation*}
		Now we find:
		\begin{align*}
			&(w_{<e},v_{<e})\cdot(-w_i, \varphi^L(w_i,w_i))\cdot(w_{>e},v_{>e})\cdot (w_i,0) \\ = & \hspace{1mm} (w_{<e},v_{<e})\cdot(-w_i, \varphi^L(w_i,w_i))\cdot(w_{>e} + w_i,v_{>e} + \varphi^L(w_{>e},w_i)) \\
			= & \hspace{1mm} (w_{<e},v_{<e})\cdot(w_{>e}, v_{>e} + \varphi^L(w_{>e},w_i) + \varphi^L(w_i,w_i) - \varphi^L(w_i, w_{>e}+w_i)) \\
			= & \hspace{1mm} (w_{<e},v_{<e})\cdot(w_{>e}, v_{>e} + \varphi(w_{>e},w_i)\\
			= & \hspace{1mm} (0, v_{<e} + v_{>e} + \varphi(w_{>e},w_i) + \varphi^L(w_{<e},w_{>e})).
		\end{align*}
		By the induction hypothesis $v_{<e} + v_{>e}  + \varphi^L(w_{<e},w_{>e})$ lies in $D$ and by the assumption on $B$, $\varphi(w_{>e},w_i) \in D$. The claim hence follows via induction.
		
		We will now argue that $N$ is a normal subgroup. Note that it is a subgroup by construction. Let $g = (w',v')$ be an arbitrary element in $G_\varphi$. For any $g \in G_\varphi$, conjugation of a generator gives us
		$$g^{-1}(w_i,0)g = (w_i, \varphi(w_i,w')) = (w_i,0)\cdot (0, \varphi(w_i,w')) \in N,$$
		since $\varphi(w_i,w') \in D$ by assumption, and
		$$g^{-1}(0,v)g = (0,v).$$
		
		As $\tau(N) = B$ and by the claim $N\cap \Z^n = D$, the statement about the index follows.
	\end{proof}
	\begin{rem}
		Although choosing another basis of $B$ might give a different normal subgroup also denoted by $N_{B,D}$, we will not specify the particular basis in future uses, since it does for example not affect the index of the subgroup nor the intersection $N_{B,D} \cap \Z^n = D$.
	\end{rem}
	\begin{rem}
		The $\Z_p$-reduction of a group $G_\varphi$, $(G_\varphi)^{\Z_p}$, is given by $G_\varphi/N_{B,D}$ with $B = p\Z^m$ and $D = p\Z^n$. In particular, there exists a homomorphism $G_\varphi \to G_\varphi^{\Z_p}$.
	\end{rem}
	\begin{ex} \label{ex_normal_subgroups}
		Consider $H_3(\Z[i])$, represented by $G_\varphi$ as in Example \ref{ex_Heis}, over $\C$, i.e. $G_\varphi^\C$, we find the bases
		$$w_1 = \begin{pmatrix}i\\1\\0\\0\end{pmatrix}, w_2 = \begin{pmatrix}0\\0\\i\\1\end{pmatrix}, w_3 = \begin{pmatrix}-i\\1\\0\\0\end{pmatrix}, w_4 = \begin{pmatrix}0\\0\\-i\\1\end{pmatrix} \subset \C^4, v_1 = \begin{pmatrix}2\\-2i\end{pmatrix}, v_2 = \begin{pmatrix}2\\2i\end{pmatrix} \subset \C^2.$$
		These base vectors satisfy
		$$\varphi^\C(w_k,w_{k+l}) = \begin{cases}
			v_1 & k=1,l=1,\\
			v_2 & k=3,l=1, \\
			0 & \text{elsewhere}.
		\end{cases} $$
		Hence, with respect to the given bases, $\varphi^\C$ is given by the matrices
		\begin{equation} \label{eq_good_basis_H3}
			\left(\begin{pmatrix}0&1&0&0\\-1&0&0&0\\0&0&0&0\\0&0&0&0\end{pmatrix},\begin{pmatrix}0&0&0&0\\0&0&0&0\\0&0&0&1\\0&0&-1&0\end{pmatrix}\right) .
		\end{equation}
		We can exploit this structure to make two normal subgroups $N_k$ ($k=1,2$) as follows: take a prime $p \equiv 1 \mod 4$. Then, there exists a number $x\in\Z$ such that $x^2 \equiv -1 \mod p$. Replacing $i\in \C$ by $x$, we obtain bases of $\Z_p^4$ and $\Z_p^2$. Now 
		$$N_k^{\Z_p} = \langle (w_{2k-1},0), (w_{2k},0), (0,v_k) \rangle, k\in \{1,2\}$$
		are normal subgroups of $G_\varphi^{\Z_p}$, or equivalently $N_{B_k,D_k}$ are normal subgroups of $G_\varphi$, where $B_k = \langle w_{2k-1}, w_{2k}, pw \mid w \in \Z^m \rangle$ and $D_k = \langle v_k, pv \mid v \in \Z^2 \rangle$, for any choice of basis of $B_k$. We can see that $|G_\varphi^{\Z_p}/N_k^{\Z_p}| = |G_\varphi/N_{B_k,D_k}| = p^3$.
	\end{ex}
	Let $\varphi$ be determined by the skew-symmetric matrices $A_1$ up to $A_n$. In the results below, we show that normal subgroups realizing $D_{G_\varphi}(0,v)$ may be assumed to be of the form $N_{B,D}$ for well-chosen $B$ and $D$. We use this to derive a formula for $D_{G_\varphi}(0,v)$.
	\begin{lemma} \label{lem_Dov_pre}
		Suppose $(0,v) \in G_\varphi$ with $v\neq 0$. Then there exists a normal subgroup of the form $N_{B,D}$ realizing $D_{G_\varphi}(0,v)$.
	\end{lemma}
	\begin{proof}
		Suppose $N$ is a normal subgroup realizing $D_{G_\varphi}(0,v)$, i.e. $(0,v) \notin N$ and $D_{G_\varphi}(0,v) = [G:N]$. Set $D = N\cap \Z^n$ and $B = \tau(N)$.
		
		Let $w \in B$, i.e. $(w,v_0) \in N$ for some $v_0 \in \Z^n$. Since $N$ is a normal subgroup, we have that $[(w,v_0),(w',0)] \in N$ for all $w' \in \Z^m$. This expression equals $(0, \varphi(w,w'))$. Hence, $\varphi(w,w') \in D$. Since $w\in B$ and $w'\in \Z^m$ are arbitrary, we conclude that $\varphi(B, \Z^m) \subset D$.
		
		Fix a basis of $B$. By Lemma \ref{lem_existence_NBD}, we conclude that $N_{B,D}$ is a well-defined normal subgroup. Furthermore, $[G_\varphi: N] = [G_\varphi:N_{B,D}]$ and since $(0,v) \notin N_{B,D}$, we conclude that $D_{G_\varphi}(0,v)$ is also realized by $N_{B,D}$.
	\end{proof}
	\begin{nota}
		For a matrix $M \in \Z^{m\times m}$ and a prime power $p^k$, we define $\Im_{\Z_{p^k}}M$ and $\ker_{\Z_{p^k}}M$ to be the image and kernel of the map $\Z^m \to \Z^m_{p^k}: w \mapsto Mw \mod p^k$.	Let $a = (a_1, \ldots , a_n)$ be a vector in $\Z^n$. We denote $\gcd(a_1, \ldots ,a_n)$ by $\gcd(a)$.
	\end{nota}
From now on, we let $\mathcal{P}$ denote the set of all primes.
	\begin{lemma}\label{lem_Dov}
		Let $(0,v) \in G_\varphi$ be a non-trivial element. The value $D_{G_\varphi}(0,v)$ equals the following minimum over all prime powers $p^k$ and all $a = (a_1, \ldots , a_n)\in \Z^n$, subject to the conditions $\ggd(a) = 1$ and $a^Tv \not\equiv 0 \mod p^k$:
		$$\min_{p\in \mathcal{P}, k\in \N, a \in \Z^n}\Big\{|\Im_{\Z_{p^k}} \sum_{i=1}^n a_iA_i |\cdot p^k \;\Big| \ggd(a) = 1, a^Tv \not\equiv 0 \mod p^k\Big\}.$$
	\end{lemma}
	\begin{proof}
		Take a normal subgroup of the form $N_{B,D}$ realizing $D_{G_\varphi}(0,v)$. Now, $\Z^n/D$ is a finite abelian group. Hence, there exists an isomorphism 
		$$\psi:  \Z^n/D \to \oplus_{j=1}^d\Z_{p_j^{k_j}}$$
		for some $d\in \mathbb{N}$ and some prime powers $p_j^{k_j}$. Since $v\notin D$, we know that $\psi(v D) \neq 0$, so there exists a summand $\Z_{p_i^{k_i}}$ such that the projection of $\psi(v D)$ onto this summand is non-zero. Let $D'$ be the preimage in $\Z^n$ of the map $\Z^n \to \oplus_{j=1}^d\Z_{p_j^{k_j}} \to \Z_{p_i^{k_i}}$. Then, $v \notin D'$ and $[\Z^n:D'] = p_i^{k_i}$. Now, by Lemma \ref{lem_existence_NBD}, $N_{B,D'}$ is still a normal subgroup not containing $(0,v)$ of index not larger than $[G_\varphi: N_{B,D}]$. Hence, we may assume that $D = D'$.
		
		By what we have shown, we may assume $D$ has index $p^k$ and $\Z^n/D \cong \Z_{p^k}$. We can now take a basis of $\Z^n$, say $\{v_1, \dots v_n\}$, such that $\{p^kv_1, v_2, \dots v_n\}$ is a basis of $D$. Let $a$ denote the vector in $\Z^n$ defining the projection onto $\langle v_1 \rangle$ with respect to the previous basis, i.e. $a^T(\sum_{i=1}^nn_iv_i) = n_1$. Remark that $\ggd(a) = 1$ and $a^Tv \not\equiv 0 \mod p^k$.
		
		Now, $B$ needs to be as large as possible satisfying the inclusion $\varphi(w', w) \subset D$ for all $w\in B$ and $w' \in \Z^m$. Using the projection $a$, the inclusion $\varphi(w', w) \subset D$ can be rephrased to $a^T\varphi(w', w) \equiv 0 \mod p^k$. This equals $w'^T(\sum_{i=1}^na_iA_i)w \in \Z_{p^k}$. Using that this holds for all $w' \in \Z^m$, it follows that
		$$ (\sum_{i=1}^na_iA_i)w \in \Z_{p^k}^m$$
		for all $w\in B$. We conclude that $B$ must equal $\ker_{\Z_{p^k}} (\sum_{i=1}^na_iA_i)$.
		
		By Lemma \ref{lem_existence_NBD}, the index of $N_{B,D}$ is $|\Im_{\Z_{p^k}} (\sum_{i=1}^na_iA_i)|\cdot p^k$. This yields the desired minimum.
	\end{proof}
	\begin{ex} \label{ex_Dov}
		Continuing on Example \ref{ex_normal_subgroups}, we see that the normal subgroup $N_{B_1,D_1}$ is of the form appearing in the definition of $D_{G_\varphi}(0,v)$. Indeed, 
		$$\Z^2/D_1 \cong \Z^2_p/\langle v_1 \rangle \cong \Z_p.$$
		The projection $a$ is chosen such that $a^Tv \neq 0$ if $v\notin D_1$. Hence, modulo $p$, this is a vector such that $a^Tv \not\equiv 0 \mod p$ if $v \notin \langle v_1 \rangle$. Therefore, $a \mod p$ can be taken to be the projection onto $\langle v_2\rangle$. Now the modulo $p$ reduction of $B$ equals 
		\begin{equation*}
			\begin{split}
				B \mod p & = \ker_{\Z_{p}} (\sum_{i=1}^na_iA_i) \mod p\\
				& = \{w \in \Z_p \mid \forall w'\in \Z_p^4: a^T\varphi(w, w') \equiv 0 \mod p\} \\
				& = \{w \in \Z_p \mid \forall w'\in \Z_p^4: \varphi(w, w') \in \langle v_1\rangle\} \\
				& = \langle w_1,w_2 \rangle.
			\end{split}
		\end{equation*}
	\end{ex}

	\subsection{Proof of the Upper Bound}
	In this subsection, we will prove that $\RF_{G_\varphi} \preceq \log^d$ for some $d$, depending on $\varphi^\C$ or equivalently depending on the complex Mal'cev completion. Throughout this section, we assume $\varphi: \Z^m\times \Z^m \to \Z^n$ is given and defined by $A_1$ to $A_n$.
	\begin{df}
		Let $\psi: \C^m\times \C^m \to \C^n$ be an alternating, bilinear map defined by the complex matrices $A_1$ up to $A_n$. We define the number $d(\psi)$ to be
		$$d(\psi) = \min\{\max_{j=1}^n\{\rank_\mathbb{C}\sum_{i=1}^n a^{(j)}_iA_i\}\mid a^{(1)} \text{ to }a^{(n)}\text{ is a basis of }\mathbb{C}^n\}.$$
	\end{df}
	By definition, there exists a basis $\{a^{(1)}, \ldots , a^{(n)}\}$ of $\C^n$ such that for every $v\in\C^n$ we can take a base vector $a^{(j)}$ such that $(a^{(j)})^Tv \neq 0$ and 
	$$\rank_\mathbb{C}\sum_{i=1}^n a^{(j)}_iA_i \leq d(\psi).$$
	In fact, such a basis does not exist if we replace $d(\psi)$ by $d(\psi)-1$. 
	\begin{rem}
		Note that $d(\psi)$ is always even, since the rank of a skew-symmetric matrix is always even. In the remark following Proposition \ref{prop_connection_upper_lower}, we will comment on how one can attempt to calculate this number.
	\end{rem}
	\begin{ex}
		Example \ref{ex_normal_subgroups} shows that $d(\varphi^\C) = 2$. Indeed, we only need to take $a^{(1)}$ and $a^{(2)}$ to be the projections onto $v_1$ and $v_2$. With respect to the basis $\{w_1,w_2,w_3,w_4\}$, the matrices $a_1^{(j)}A_1+a_2^{(j)}A_2$ ($j\in\{1,2\}$)  are then given in Equation \eqref{eq_good_basis_H3}. These have rank $2$, and this is the smallest possible.
	\end{ex}

	The goal is to prove the following theorem:
	\begin{thm} \label{thm_upper_bound}
		Let $G_\varphi$ be a two-step $\mathcal{I}$-group defined by $\varphi:\Z^m\times \Z^m \to \Z^n$. We have $\RF_{G_\varphi} \preceq \log^{d(\varphi^\C)+1}$.
	\end{thm}
	In particular, this result tells us that $\RF_{G_\varphi}$ has an upper bound depending on $G_\varphi^\C$ as we remark below.
	\begin{rem}
		Let $P\in \GL(n, \C)$ and $Q\in \GL(m, \C)$. Suppose $\psi$ and $\psi'$ are two alternating, bilinear maps related by the formula 
		\begin{equation} \label{eq_equiv_C_compl}
		\psi(w,w') = P\psi'(Qw,Qw')
		\end{equation}
		for all $w,w'\in \C^m$. The matrices $P$ and $Q$ represent a change of basis in $\C^m$ and $\C^n$ respectively. Therefore, one easily sees that $d(\psi) = d(\psi')$.  
		
		If $G_{\varphi}$ and $G_{\varphi'}$ have isomorphic complex completions, then $\varphi^\C$ and $(\varphi')^\C$ satisfy a relation of the form given in Equation \eqref{eq_equiv_C_compl}, see for example the argumentation in \cite[Lemma 1]{grunewald1982nilpotent}. Therefore, the upper bound of Theorem \ref{thm_upper_bound} is a well-defined invariant of the $\C$-completion.
	\end{rem}
	
	To prove the theorem, we will need the following result, given in \cite[Corollary 3.11 and its remarks]{serre2016lectures}, but first established in \cite{ax1967solving,van1991remark}:
	\begin{thm}
		Let $l, s \in \mathbb{N}$. Suppose $V = \{q_j(x_1, \dots , x_l)\mid 1\leq j \leq s\}$ is a finite set of polynomials in $\Z[x_1, \dots , x_l]$. If these polynomials have a common zero over $\mathbb{C}$, then the set $P\subset \mathcal{P}$ of primes such that $V$ has a common zero over $\Z_p$ has density
		$$\pi_P(x) \asymp \dfrac{x}{\log(x)}.$$
	\end{thm}
	Here, $\pi_P(x)$ is the prime counting function $|\{p \in P\mid p \leq x \}|$, and $\pi_P(x) \asymp x/\log(x)$ means that there exist $C_1,C_2>0$ such that
	$$C_1\dfrac{x}{\log(x)} \leq \pi_P(x)\leq C_2\dfrac{x}{\log(x)} $$
	for all $x$ sufficiently large. By \cite[Proposition 4.7]{dere2023residual}, this implies that $C>0$ exists such that for all $0\neq m \in \Z$ there exists a prime $p\in P$ with $m \notin p\Z$ and 
	\begin{equation} \label{eq_density}
	p \leq C\cdot \log(|m|) + C.
	\end{equation}
	\begin{cor} \label{cor_minimal_bases}
	The set of primes $P\subset \mathcal{P}$ defined by the property
		$$\min\{\max_{j=1}^n\{\rank_{\mathbb{Z}_p}\sum_{i=1}^n a^{(j)}_iA_i\}\mid a^{(1)} \text{ to }a^{(n)}\text{ is a basis of }\Z_p^n\} \leq d(\varphi^\C)$$
		has density
		$$\pi_P(x) \asymp \dfrac{x}{\log(x)}.$$
	\end{cor}
	\begin{proof}
		By the definition of $d(\varphi^\C)$, we know that a basis realizing this bound exists over $\mathbb{C}$. In order to apply the previous theorem, it suffices to argue that this bound can be expressed as multivariate polynomials.
		
		Note that this bound means that there exists a basis $a^{(1)}$ up to $a^{(n)}$ of $\Z_p^n$ such that 
		$$\forall j \in \{1, \dots , n\}: \rank\sum_{i=1}^n a^{(j)}_iA_i \leq d(\varphi^\C).$$
		The existence of a basis can be expressed as
		$$\det(a^{(1)}, \dots , a^{(n)})\cdot \lambda - 1= 0.$$
		The bound on the rank can be expressed via minors, i.e. determinants of square submatrices. Indeed, saying that the rank is at most $d(\varphi^\C)$ is equivalent with saying that all minors of larger size are zero:
		$$\forall j \in \{1, \dots , n\}:\forall k>d(\varphi^\C): \forall (k\times k)\text{-minors } q \text{ of } \sum_{i=1}^n a^{(j)}_iA_i: q = 0.$$
		These expressions form a finite list of integer polynomials with variables $a^{(j)}_i$ and $\lambda$. 
	\end{proof}
	\begin{ex}
		In Example \ref{ex_normal_subgroups}, we constructed two normal subgroups $N^{\Z_p}_1$ and $N^{\Z_p}_2$. Recall that they are both defined via the basis $\{v_1,v_2\}$ of $\Z_p^2$. Corresponding to this basis, we find the basis $\{a^{(1)},a^{(2)}\}$ consisting of the projections onto $\{v_1,v_2\}$. In Example \ref{ex_Dov}, we illustrated that 
		$$\rank_{\mathbb{Z}_p}\Im\sum_{i=1}^n a^{(2)}_iA_i = \rank_{\Z_p} \langle w_3,w_4 \rangle = 2,$$
		as the kernel consisted of $\langle w_1,w_2\rangle$. We find the same rank for $a^{(1)}$, showing that
		$$\min\{\max_{j=1}^n\{\rank_{\mathbb{Z}_p}\sum_{i=1}^n a^{(j)}_iA_i\}\mid a^{(1)} \text{ to }a^{(n)}\text{ is a basis of }\Z_p^n\} \leq d(\varphi^\C) = 2$$
		at least for all primes $p\equiv -1 \mod 4$. In terms of density, this accounts for `half of all primes'.
	\end{ex}
	Recall that $D(0,v)$ equals
	$$\min_{p\in\mathcal{P}, k\in\N,a\in \Z^n }\Big\{|\Im_{\Z_{p^k}} \sum_{i=1}^n a_iA_i |\cdot p^k \;\Big| \ggd(a) = 1, a^Tv \not\equiv 0 \mod p^k\Big\}.$$
	In a first reduction step, we overestimate this minimum by only looking at a subset of primes $P\subset \mathcal{P}$, rather than all prime powers $p^k$.
	\begin{lemma} \label{lem_redu_to_primes}
		Take notations as in Lemma \ref{lem_Dov} and let $P$ be a subset of primes. We have
		$$D_{G_\varphi}(0,v) \leq \min\{p^{1+\rank_{\Z_p}\sum_{i=1}^n a_iA_i}\mid p\in P: a\in \Z_p^n: a^Tv \not\equiv 0 \mod p\}.$$
	\end{lemma}
	\begin{proof}
		In the minimum of $D(0,v)$, we are taking the minimum over all prime powers $p^k$. Restricting the minimum to prime numbers in $P$ will result in a larger number:
		$$D(0,v) \leq \min_{p\in P, a\in\Z^n}\Big\{|\Im_{\Z_{p}} \sum_{i=1}^n a_iA_i |\cdot p \Big| \ggd(a) = 1, a^Tv \not\equiv 0 \mod p\Big\}.$$
		The right hand side equals
		$$\min_{p\in P, a\in \Z^n}\Big\{p^{1+\rank_{\Z_p}\sum_{i=1}^n a_iA_i}\;\Big| \ggd(a) = 1, a^Tv \not\equiv 0 \mod p\Big\}.$$
		We claim this equals
		$$\min\{p^{1+\rank_{\Z_p}\sum_{i=1}^n a_iA_i}\mid p\in P: a\in \Z_p^n: a^Tv \not\equiv 0 \mod p\}.$$
		Indeed, $\ggd(a) = \ggd(a_1, \ldots , a_n)$ cannot be a multiple of $p$, because otherwise $a^Tv \equiv_p 0$. Suppose $\ggd(a) = b$. We just argued that $b$ is invertible in $\Z_p$, so $$\rank_{\Z_p}\sum_{i=1}^n a_iA_i = \rank_{\Z_p}\sum_{i=1}^n b^{-1}a_iA_i,$$
		$(b^{-1}a)^Tv \not\equiv 0 \mod p$ and $\ggd(a_1/b, \dots a_n/b) = 1$. In conclusion, any case where $\ggd(a) \neq 1$ reduces to a case where it is with the same rank/outcome.
	\end{proof}
	In the proof of Theorem \ref{thm_upper_bound}, we will use the estimate above for primes $P$ coming from Corollary \ref{cor_minimal_bases}. They satisfy the density property $\pi_P(x) \asymp x/\log(x)$. Given $(0,v) \in G_\varphi$, Equation \eqref{eq_density} allows us to take a prime $p\in P$ with $p \preceq \log(\norm{v}_\infty)$ such that $v \notin p\Z^n$. Here, $\norm{v}_\infty = \max\{|v_i|\}$. The following lemma tells us that if $(0,v) \in B_{G_\varphi}(r)$, where we take the standard generators of (the set) $\Z^m\times \Z^n$, then $\norm{v}_\infty \preceq r^2$. Hence, the chosen prime would satisfy $p \preceq \log(r)$. 
	\begin{lemma} \label{lem_2step_metric}
		There exists a constant $C>0$ such that if $(w,v) \in B_{G_\varphi}(r)$, then $\norm{w}_\infty \leq r$ and $\norm{v}_\infty \leq Cr^2$ with $\norm{v}_\infty = \max\{|v_i|\}$.
	\end{lemma}
	\begin{proof}
		Let $(w,v) \in B_{G_\varphi}(r)$. Define $C$ to be  $\max\{\norm{\varphi^L(e_i,e_j)}_\infty\mid 1\leq i,j \leq m\}$. By definition, the element $(w,v)$ can be written as
		$$(w,v) = \left(\prod_{j=1}^{r_1} (e_{i_j},0)^{\epsilon_j}\right) \cdot \left( \prod_{l=1}^{r_2} (0, e_{k_l})^ {\epsilon_l}\right),$$
		with $1 \leq i_j \leq m$, $1\leq k_l\leq n$, $\epsilon_s \in \{1,-1\}$ and $r_1+r_2 \leq r$.
		Since $(e_{i},0)^{-1} = (-e_i, \varphi^L(e_i,e_i))$ and $(\epsilon_i e_i,v)\cdot (w',v') = (\epsilon_i e_i+w',v+v'+\epsilon_i\varphi^L(e_i,w))$, we can (inductively) conclude that
		$$\prod_{j=1}^{r_1} (e_{i_j},0)^{\epsilon_j} = (\sum_{j=1}^{r_1}\epsilon_je_{i_j}, \sum \epsilon_{i_j,i_{j'}}
		\varphi^L(e_{i_j},e_{i_{j'}})),$$
	where $\epsilon_s \in \{1,-1\}$ and the second sum has at most $r_1 + r_1(r_1-1)/2 \leq r_1^2$ terms ($r_1\geq 1$). Indeed, one can have at most $r_1$ terms coming from inverted generators and at most $1+2+\ldots + (r_1-1)$ terms from multiplication. Now, it is clear that $\norm{w}_\infty \leq r_1 \leq r$ and
		$$\norm{v}_\infty \leq \sum\norm{\varphi^L(e_{i_j},e_{i_{j'}})}_\infty + \sum_{l=1}^{r_2} \norm{e_{k_l}}_\infty \leq C(r_1)^2 + r_2 \leq Cr^2.$$
	\end{proof}
	We are now ready to prove the upper bound.
	\begin{proof}[Proof of Theorem \ref{thm_upper_bound}]
		Let $(w,v) \in B_{G_\varphi}(r)$.
		We have either $w \neq 0$ or $w = 0$. If $w\neq 0$, then $w$ remains non-trivial under the homomorphism $G_\varphi \to \Z^m$. By the residual finiteness of $\Z^m$, we find some bound
		$$D(w,v) \preceq \log(r)$$
		for all those elements $(w,v)$ with $w\neq 0$.
		
		Write $P$ for the set of primes of Corollary \ref{cor_minimal_bases}. By Lemma \ref{lem_redu_to_primes}, we know that
		\begin{equation*}
				D(w,v) \leq \min\{p^{1+\rank_{\Z_p}\sum_{i=1}^n a_iA_i}\mid p\in P: a\in \Z_p^n: a^Tv \not\equiv 0 \mod p\}.
		\end{equation*}
		In what follows, we will reduce the $a\in \Z_p^n$ over which the minimum runs. Indeed, for the primes $p\in P$, we can set $a^{(1)}$ to $a^{(n)}$ to denote the basis realizing the minimum of Corollary \ref{cor_minimal_bases} for that prime. Now restrict the possible projections $a$ to this basis and use that the rank can be bounded by $d(\varphi^\C)$ for this basis. We obtain
		\begin{equation*}
			\begin{split}
				D(w,v) & \leq \min\{p^{1+\rank_{\Z_p}\sum_{i=1}^n a_iA_i}\mid p\in P: a\in \{a^{(1)}, \dots a^{(n)}\}: a^Tv \not\equiv 0 \mod p\}\\
				& \leq \min\{p^{1+d(\varphi^\C)}\mid p\in P: a\in \{a^{(1)}, \dots a^{(n)}\}: a^Tv \not\equiv 0 \mod p\}\\
				& \leq \min\{p^{1+d(\varphi^\C)}\mid p\in P: v \not\equiv 0 \mod p\}\\
				& \preceq \log^{1+d(\varphi^\C)}(r^2) \approx \log^{1+d(\varphi^\C)}(r).
			\end{split}
		\end{equation*}
		In the last steps, we used that $v \equiv 0$ if and only if $a^Tv\equiv 0$ for all $a\in \{a^{(1)}, \dots a^{(n)}\}$. Then, we applied the density result of the set $P$ (see Corollary \ref{cor_minimal_bases} and Equation \eqref{eq_density}) and Lemma \ref{lem_2step_metric}.
		
		All together, for all $(w,v) \in B_G(r)$, we find the bound
		$$D(w,v) \preceq \log^{d(\varphi^\C)+1}(r).$$
		This ends the proof.
	\end{proof}
	\begin{ex}
		In Example \ref{ex_normal_subgroups} and \ref{ex_Dov}, we constructed two normal subgroups $N^{\Z_p}_1$ and $N^{\Z_p}_2$. Recall that they are both defined via a basis $\{v_1,v_2\} \subset \Z_p^2$. The proof of Theorem \ref{thm_upper_bound} boils down to the following reasoning:
		\begin{enumerate}
			\item Given an arbitrary element $(0,v) \in G_\varphi$, take a prime $p \equiv -1 \mod 4$ such that $v\not\equiv 0 \mod p$. Now the image of $(0,v)$ is non-trivial under the map
			$$G_\varphi \to G_\varphi^{\Z_p} .$$
			\item The vector $v\mod p$ cannot lie in both $\langle v_1 \rangle$ and $\langle v_2 \rangle$ as they form a basis together. Suppose $v \mod p$ does not lie in $\langle v_1 \rangle$. Then the image of $(0,v)$ remains non-trivial under the composition
			$$G_\varphi \to G_\varphi^{\Z_p} \to G_\varphi^{\Z_p}/N^{\Z_p}_1 .$$
			This shows that $D(0,v) \leq p^3$.
			In the other case, one can use $N^{\Z_p}_2$ yielding $D(0,v) \leq p^3$ as well.
			\item Taking into account the density of primes $p\equiv -1\mod 4$, we know $p$ can be chosen proportional to $\log(\norm{v})$.
		\end{enumerate}
	We conclude that $\RF_{H_3(\Z[i])}$ equals $\log^3$, which improves the $\log^4$ upper bound of \cite[Proposition 2.2]{pengitore2018corrigendum}.
	\end{ex}
	\begin{rem}
	In Lemma \ref{lem_Dov} and Theorem \ref{thm_upper_bound}, the normal subgroups are obtained by reducing $G_\varphi$ modulo $p$ (or $p^k$). Then, we make sure we have a one-dimensional center $\Z^n/D$ in our quotient.
	
	One can also do the converse: make sure that $\Z^n/D$ is a one-dimensional $\Z$-module, and then reducing modulo $p$. This will not give optimal bounds. For example in $H_3(\Z[i])$ (see Example \ref{ex_Heis}), we have
	$$G_\varphi \to \dfrac{G_\varphi}{\langle (0,(1,0)) \rangle} \cong G'_{\varphi'},$$
	with $G'_{\varphi'}$ given by the set $\Z^4\times \Z$ with 
	$$\varphi' = \begin{pmatrix}0&0&0&-1\\0&0&-1&0\\0&1&0&0\\1&0&0&0\end{pmatrix}.$$
	As we have seen $\RF_{G_\varphi} = \log^3$, but $\RF_{G'_{\varphi'}} = \log^5$, since the matrix above has rank $4$ (see Lemma \ref{lem_one_matrix}). In fact, the same is true if we replace $(1,0)$ by any $v\in \Z^2$.
	\end{rem}
	\begin{rem}
		In the reduction steps of Theorem \ref{thm_upper_bound}, the minimum of $D(0,v)$ as in Lemma \ref{lem_Dov} seems largely overestimated. Yet, as will also be remarked in Section \ref{sec_lower}, no examples are known where this bound is not exact.
	\end{rem}
	
	\section{Comparison to the Previous Bound}
	\label{previousbound}
	In this section, we will demonstrate that our new upper bound improves the previously known bound in literature. In fact, the non-optimality of the previous bound was already communicated in \cite{pengitore2018corrigendum}. However, we show that the difference between the previous and the new bound can get arbitrarily large. This is done by using a class of $\mathcal{I}$-groups defined by non-singular (two-step nilpotent) Lie algebra.
	
	\subsection{Previous Bound and Non-Singularity}
	We start by reintroducing the bound given in \cite{pengitore2018corrigendum}. Then, we will proceed to introduce the notion of non-singularity. We can calculate the previous bound for non-singular $\mathcal{I}$-groups over $\Q$. In the next subsection, we will demonstrate that this bound is not an optimal predictor of $\RF_G$ in this subclass.
	\begin{df}
		Given a nilpotent group $G$, given a non-trivial, central element $g$, define $\psi(g)$ as
		$$\min\{h(G/N)\mid G/N \text{ irr, t.f.}, e\neq \langle\pi_N(g)\rangle \leq_f Z(G/N)\}. $$
		Here, irr. stands for irreducible, i.e. the quotient $G/N$ is no direct sum of groups. The abbreviation t.f. means torsion-free.
		
		Let $\psi(G)$ be the number
		$$\max\{\psi(g)\mid e\neq g \in Z(G), \exists k\in \N, \exists a \in G_{c-1}, \exists b \in G: g^k = [a,b]\} ,$$
		where $G_{c-1}$ determines the $c-1$'th term of the lower central series.
	\end{df}
	In \cite[Theorem 3.1]{pengitore2018corrigendum}, the following bound was stated.
	\begin{thm}
		Let $G$ be an $\mathcal{I}$-group, then $\RF_G \preceq \log^{\psi(G)}$.
	\end{thm}
	Now recall the following definition:
	\begin{df}
		Let $\mathfrak{g}$ be a two-step nilpotent Lie algebra over $\mathbb{R}$. Then $\mathfrak{g}$ is called \emph{non-singular} if $\ad (X): \mathfrak{g} \to \mathfrak{z}(\mathfrak{g}): Y \mapsto [X,Y]$ is surjective for all non-central $X$, where $\mathfrak{z}(\mathfrak{g})$ is the center of $\mathfrak{g}$.
	\end{df}
	A non-singular Lie algebra is also called fat or regular.
	\begin{df}
		A two-step nilpotent Lie group (over $\R$) is called \emph{non-singular} if its Lie algebra is non-singular.
	\end{df}
	
	Given $G_\varphi$, we wish to relate non-singularity with some property of $\varphi$. In order to do so, we will first determine the Lie algebra corresponding to $G_\varphi$.
	\begin{lemma} \label{lem_Lie_alg_twostep}
		Suppose $\varphi: \Z^m\times \Z^m \to \Z^n$ is a full, alternating, bilinear map. Let $\mathfrak{g}$ denote the Lie algebra over $\mathbb{Q}$ defined as $\Q^m\times \Q^n$ with Lie bracket $[(w,v),(w',v')] = (0, \varphi^\Q(w,w'))$. Then, $\mathfrak{g}$ is the Lie algebra corresponding to the Mal'cev completion $G_\varphi^\Q$ of $G_\varphi$ (up to isomorphism).
	\end{lemma} 
	\begin{proof}
		The Lie group $G^\Q$ corresponding to $\mathfrak{g}$ has its operation defined by the Baker-Campbell-Hausdorff formula (for $2$-step nilpotent groups):
		$$(w,v)\ast (w',v') = (w,v) + (w',v') + \frac{1}{2}[(w ,v),(w',v')] = (w + w^\prime,v + v^\prime +\frac{1}{2}\varphi^\Q(w,w')) .$$
		The group $G_\varphi$ injects into this group via the homomorphism:
		\begin{equation} \label{eq_embedd_lie_alg}
		i: G_\varphi \to G^\Q: (w,v) \mapsto (w,v-\frac{1}{2}\varphi^L(w,w)).
		\end{equation}
		Indeed, it is a homomorphism, because
		\begin{equation*}
			\begin{split}
				i\left((w,v)\cdot(w',v')\right) & =  i(w+w',v+v'+\varphi^L(w,w'))\\
				& = (w+w',v+v'+\varphi^L(w,w')-\frac{1}{2}\varphi^L(w+w',w+w'))\\
				& = (w+w',v+v'+\varphi^L(w,w')-\frac{1}{2}\varphi^L(w,w)-\frac{1}{2}\varphi^L(w,w')-\frac{1}{2}\varphi^L(w',w)-\frac{1}{2}\varphi^L(w',w')) \\
				& = (w+w',v+v'+\frac{1}{2}\varphi(w,w')-\frac{1}{2}\varphi^L(w,w)-\frac{1}{2}\varphi^L(w',w')) \\
				& = (w,v-\frac{1}{2}\varphi^L(w,w))\ast (w',v'-\frac{1}{2}\varphi^L(w',w')) \\
				& = i(w,v) \ast i(w',v'),
			\end{split}
		\end{equation*}
	and it is injective, because $(w, v-\frac{1}{2}\varphi^L(w,w)) = 0$ clearly implies that $(w,v) = 0$.
	
	One can verify that $G^\Q$ is a torsion-free radicable nilpotent group (containing $i(G)$), such that every element of $G^\Q$ has some power lying in $i(G)$. For example, for radicability we have $(w,v)^k = (kw,kv)$ for $k\in \Z$ and therefore $(w,v)^q = (qw,qv)$ for $q\in \Q$. For the last condition, if $(w,v)\in G^\Q$ is given, one takes $k\in \Z$ such that $kw$ and $kv$ are integral vectors. Now, $w' = 2kw$ and $v' = 2kv + 2\varphi^L(kw,kw)$ are integral vectors and
	$$(w,v)^{2k} = (2kw,2kv) = (w',v'-\frac{1}{2}\varphi^L(w',w')) = i(w',v') \in i(G_\varphi).$$
	By \cite[p.107]{sega83-1}, these properties mean that $G^\Q$ is the rational Mal'cev completion of $i(G_\varphi) \cong G_\varphi$. This ends the proof.
	\end{proof}
	\begin{rem}
		Note that the homomorphism in Equation \eqref{eq_embedd_lie_alg} can be extended to an isomorphism from $G_\varphi^\Q = (\Q^m\times \Q^n, \cdot)$ to $G^\Q$, where $G_\varphi^\Q$ was defined in the remark following Example \ref{ex_Heis}. This shows that $G_\varphi^\F$ is indeed the $\F$-completion of $G_\varphi$ as claimed there.
	\end{rem}
	Now, we can find the relation between $\varphi$ and non-singularity.
	\begin{lemma}
		Let $G_\varphi$ be a two-step $\mathcal{I}$-nilpotent group. The real Mal'cev completion $G_\varphi^\mathbb{R}$ is non-singular if and only if the map $\varphi^\R(w, \ast): \R^m \to \R^n: w' \mapsto \varphi^\R(w,w')$ has full rank for all $0\neq w\in\R^m$. 
	\end{lemma}

	\begin{proof}
		By definition and by Lemma \ref{lem_Lie_alg_twostep}, the non-singularity of $G_\varphi^\mathbb{R}$ means that $\mathfrak{g}^\mathbb{R}$ is non-singular. Hence, we need to verify that non-singularity of $\mathfrak{g}^\mathbb{R}$ is equivalent with $\varphi^\R(w, \ast)$ having full rank for all $0\neq w\in\R^m$. Note that the center of $\mathfrak{g}^\R$ is always at least $(0, \R^n)$.
		
		Suppose first that $\mathfrak{g}^\R$ is non-singular. Let $(w,v) \in \mathfrak{g}^\R$ be non-central, for $(w',v') \in \mathfrak{g}^\R$ we find by definition that
		\begin{equation} \label{eq_singular}
		(w',v') \mapsto [(w,v),(w',v')] = (0, \varphi^\R(w,w'))
		\end{equation}
		is a surjection onto the center. Hence, the center must equal $(0, \R^n)$ and $w' \mapsto \varphi(w,w')$ has to be surjective, i.e. it has full rank, if $w \neq 0$.
		
		For the converse, if $w'\mapsto \varphi(w,w')$ is surjective for all $w\neq 0$, then an element of the form $(w,v)$ with $w\neq 0$ cannot be central. Hence, the center is $(0, \mathbb{R}^n)$, and for a non-central $(w,v)$, Equation \eqref{eq_singular} gives a surjection onto the center. We conclude that $\mathfrak{g}^\R$ is non-singular.
	\end{proof}

	This result gives rise to the following definition of non-singularity over fields different from $\R$.
	\begin{df}
		We say that $G_\varphi$ is non-singular over the field $\mathbb{F}$ if $\varphi^\mathbb{F}(w, \ast): \mathbb{F}^m \to \mathbb{F}^n: w' \mapsto \varphi(w,w')$ has full rank for all $w\in\mathbb{F}^m$. 
	\end{df}
	\begin{rem}
		By definition, $G_\varphi$ is non-singular over $\R$ if and only if $(G_\varphi)^\R$ is non-singular. Note that the non-singularity of the Lie group $(G_\varphi)^\R$ implies non-singularity of $G_\varphi$ over $\mathbb{Q}$.
	\end{rem}
	We will show in the remark following Proposition \ref{prop_result_Galois_class} that a two-step $\mathcal{I}$-group can be non-singular over $\F$ but can fail to be so over a field extension of $\F$.
	
	We end this subsection by calculating the residual finiteness growth for two-step $\mathcal{I}$-groups which are non-singular over $\Q$.
	\begin{lemma} \label{lem_bound_nonsing}
		Let $G_\varphi$ be two-step nilpotent. If $G_\varphi$ is non-singular over $\Q$, then $\psi(G)$ equals $m+1$.
	\end{lemma}
	\begin{proof}
		The elements $(w,v)$ that lie in the center and for which there exists $a\in G_{c-1}$ and $b\in G$ such that $(w,v)^k = [a,b]$ for some $k\in\N$ are elements of the form $(0,v)$ with $v\in\Z^n$.
		
		Take any element $g=(0,v)$. We show that 
		$$\min\{h(G/N)\mid G/N \text{ irr, t.f.}, e\neq \langle\pi_N(g)\rangle \leq_f Z(G/N)\} $$
		equals $m+1$. This then ends the proof.	Take $N$, realizing the minimum above. If $(w, v_2)$ were an element of $N$ with $w\neq 0$, then 
		$[(w,v_2), (w^\prime, 0)] = (0, \varphi(w,w^\prime)) \in N$ for all $w^\prime \in \Z^m$. Since $\varphi^\Q(w, \ast)$ has full rank, the image of $w^\prime \mapsto \varphi(w,w^\prime)$ is a finite index subgroup of $\Z^n$. Finally, note that $G/N$ has to be torsion-free, so $(0, \Z^n) \subset N$. This contradicts the assumption that $e\neq \langle \pi_N(g)\rangle$. Hence, $N\subset (0, \Z^n)$. Since the quotient has one-dimensional center $Z(G/N)$, this dictates that $h(G/N) \geq  m+1$.
	\end{proof}

	\subsection{Construction of Proposition \ref{prop_psiG_dG}}
	We now proceed to show that the class of non-singular groups (over $\Q$) contains groups where the actual residual finiteness growth is $\log^3$ (in agreement with our new bound). The upper bound $\log^{m+1}$ previously established is therefore not optimal. The techniques closely follow the methodology from \cite{dd14-1}.
	\begin{df}
		Let $k_1, k_2, k_3, k_4 \in \N$. The Kronecker product $A\otimes B$ of two matrices $A\in \C^{k_1\times k_2}$ and $B\in \C^{k_3\times k_4}$ is given by
		$$A \otimes B = \begin{pmatrix}a_{11} B & a_{12} B & \dots \\ a_{21} B & a_{22} B & \dots \\ \vdots & \vdots & \ddots \end{pmatrix} \in \C^{k_1k_3\times k_2k_4}.$$
	\end{df}
	Let $A_i \in \Z^{n\times n}$ be the diagonal matrix with zero everywhere except $1$ on the $(i,i)$-position. The bilinear map
	$$\rho: \Z^{2n}\times \Z^{2n} \to \Z^n: (w,w') \mapsto \begin{pmatrix} w^T\left(A_1\otimes \begin{psmallmatrix}0&1\\-1&0\end{psmallmatrix}\right)w' \\ \vdots \\ w^T\left(A_n\otimes \begin{psmallmatrix}0&1\\-1&0\end{psmallmatrix}\right)w'\end{pmatrix} $$
	yields the direct product $G_\rho = H_3(\Z) \oplus \ldots \oplus H_3(\Z)$ of $n$ discrete Heisenberg groups. In particular, its residual finiteness growth is $\log^3$. In what follows, we will perform base changes on $\F^{2n}$ and $\F^n$ to obtain a new bilinear map $\varphi_\F : \Z^{2n} \times \Z^{2n} \to \Z^n$ such that $G_{\varphi_\F}$ is non-singular over $\Q$ (and thus $\psi(G_{\varphi_\F})= 2n+1$). However,
	$$G_{\varphi_\F}^\F \cong G_\rho^\F,$$
	so its residual finiteness growth remains $\log^3$.
	
	We need the following notation to define the base change:
	\begin{df}
		Let $\mathbb{F}$ be a number field, Galois over $\Q$. Let $[\mathbb{F}:\Q] = n$ and $\text{Gal}(\mathbb{F}/\Q) = \{\sigma_1, \dots , \sigma_n\}$. Define $K_i \in \GL(n, \Z)$ to be the permutation matrix corresponding to the action of $\sigma_i$ on $\{\sigma_1, \dots , \sigma_n\}$, i.e.
		$$(K_i)_{k,l} = 1 \Leftrightarrow \sigma_i\circ \sigma_k = \sigma_l$$ and $0$ otherwise.
	\end{df}
	Note that $\mathbb{F} = \Q(\alpha)$ for some algebraic integer $\alpha \in \mathcal{O}_\mathbb{F}$, as number fields are primitive. This way we can define a matrix $E \in \GL(n, \mathbb{F})$ via
	$$(E)_{k,l} =  \sigma_k(\alpha^l).$$
	By construction, we have $\sigma_i(E) = K_iE$.
	\begin{lemma}
		Let $A_i$ be the diagonal matrix with zero everywhere except $1$ on the $(i,i)$-position. The bilinear map
		$$\varphi_\mathbb{F}: \Z^{2n}\times \Z^{2n} \to \Z^n: (w,w') \mapsto E^T\begin{pmatrix} w^TM_1w' \\ \vdots \\ w^TM_nw'\end{pmatrix} $$
		with $M_i = (E\otimes \mathbb{1}_2)^T(A_i\otimes \begin{psmallmatrix}0&1\\-1&0\end{psmallmatrix})(E\otimes \mathbb{1}_2)$ is well-defined.
	\end{lemma}
	\begin{proof}
		Denote $D\otimes \mathbb{1}_2$ by $D^{\otimes 2}$. The map $\varphi_\mathbb{F}$ is defined by $\varphi_{\mathbb{F}}(w,w') = (w^TB_1w', \dots , w^TB_nw')^T$ with 
		$$B_j = \sum_{i=1}^n \sigma_i(\alpha^j) (E^{\otimes 2})^T(A_i\otimes \begin{psmallmatrix}0&1\\-1&0\end{psmallmatrix})E^{\otimes 2}.$$
		We need to verify that all $B_j$ are integral matrices. Take $\sigma_k \in \text{Gal}(\mathbb{F}/\Q)$, then we find that
		$$\sigma_k(B_j) = \sum_{i=1}^n (\sigma_k\circ\sigma_i)(\alpha^j) \sigma_k(E^{\otimes 2})^T(A_i\otimes \begin{psmallmatrix}0&1\\-1&0\end{psmallmatrix})\sigma_k(E^{\otimes 2}).$$
		Since $\sigma_k(E) = K_kE$, we have $\sigma_k(E^{\otimes 2}) = K_k^{\otimes 2}E^{\otimes 2}$. We obtain
		$$\sigma_k(B_j) = \sum_{i=1}^n (\sigma_k\circ\sigma_i)(\alpha^j) (E^{\otimes 2})^T(K_k^{\otimes 2})^T(A_i\otimes \begin{psmallmatrix}0&1\\-1&0\end{psmallmatrix})K_k^{\otimes 2}E^{\otimes 2}.$$
		If $\sigma_k\circ\sigma_i = \sigma_l$, then we see that $(K_k^{\otimes 2})^T(A_i\otimes \begin{psmallmatrix}0&1\\-1&0\end{psmallmatrix})K_k^{\otimes 2}$ equals $A_l\otimes \begin{psmallmatrix}0&1\\-1&0\end{psmallmatrix}$. Hence, we have
		$$\sigma_k(B_j) = \sum_{l=1}^n \sigma_l(\alpha^j) (E^{\otimes 2})^T(A_l\otimes \begin{psmallmatrix}0&1\\-1&0\end{psmallmatrix})E^{\otimes 2} = B_j.$$
		Since $\sigma_k$ was arbitrary, and $\mathbb{F}$ is Galois over $\Q$, we conclude that $B_j$ is a rational matrix. Furthermore, as it is built from algebraic integers only, it has integral entries.
	\end{proof}
	\begin{prop} \label{prop_result_Galois_class}
		Let $G$ denote the two-step nilpotent group with bilinear map $\varphi_\mathbb{F}$. Then 
		$\RF_{G} = \log^3$ and $\psi(G) = 2n+1$.
	\end{prop}
	\begin{proof}
		The bilinear map $\varphi_\mathbb{F}$ is defined by the matrices $B_1$ up to $B_n$. By the remark following Theorem \ref{thm_upper_bound}, we know that $d(\varphi^\C)$ equals $d(\psi)$ with $\psi: \C^m\times \C^m \to \C^n$ defined via the matrices $A_i\otimes \begin{psmallmatrix}0&1\\-1&0\end{psmallmatrix}$. The value $d(\psi)$ is easily seen to be $2$, using the standard projections of $\C^n$. Hence, by Theorem \ref{thm_upper_bound}, we have $\RF_G \preceq \log^3$. Since the Heisenberg group embeds in every two-step $\mathcal{I}$-group and its residual finiteness growth equals $\log^3$, we know that in fact $\RF_G = \log^3$.
		
		For the second statement, it suffices to show that $\varphi_\mathbb{F}(w, \ast)$ has full rank for all $0\neq w\in\Q^{2n}$. The conclusion then follows by Lemma \ref{lem_bound_nonsing}.
		
		Assume by contradiction that there is a vector $0\neq w\in \Q^{2n}$ such that $\varphi_\mathbb{F}(w, \ast)$ does not have full rank. In particular, this means there is some projection $\pi:\Q^n \to \Q$ defined by $(a_1, \dots , a_n)$ such that $\pi\circ\varphi_\mathbb{F}(w,w') = 0$ for all $w'\in \Q^{2n}$. This expression equals
		$$0 = \begin{pmatrix}a_1 & \ldots & a_n\end{pmatrix}\begin{pmatrix}w^TB_1w'\\ \vdots \\ w^TB_nw'\end{pmatrix} = w^T\left(\sum_{i=1}^na_iB_i\right)w'.$$
		We claim that $ \sum_{i=1}^na_iB_i \in \Q^{2n\times 2n}$ has full rank for any non-zero choice of $(a_1 , \dots , a_n) \in \Q^n$. If so, the statement above is not possible for all $w'\in \Q^{2n}$, yielding a contradiction to the assumption, which ends the proof.
		
		We will check that the matrix $\sum_{i=1}^na_iB_i$ has a non-zero determinant. Recall we have
		\begin{equation*}
			\begin{split}
				\sum_{i=1}^na_iB_i &=\sum_{i=1}^na_i\sum_{j=1}^n \sigma_j(\alpha^i) (E^{\otimes 2})^T(A_j\otimes \begin{psmallmatrix}0&1\\-1&0\end{psmallmatrix})E^{\otimes 2} \\
				& =(E^{\otimes 2})^T\left( \sum_{j=1}^n\sum_{i=1}^n a_i\sigma_j(\alpha^i) (A_j\otimes \begin{psmallmatrix}0&1\\-1&0\end{psmallmatrix})\right)E^{\otimes 2}.  
			\end{split}
		\end{equation*}
		Hence, as $\det(E) \neq 0$, the determinant is non-zero if and only if
		$$0 \neq \det\left(\sum_{j=1}^n\sum_{i=1}^n a_i\sigma_j(\alpha^i) (A_j\otimes \begin{psmallmatrix}0&1\\-1&0\end{psmallmatrix})\right).$$
		This matrix is block-diagonal, whose determinant is easily seen to be
		$$\prod_{j=1}^n\left(\sum_{i=1}^n a_i\sigma_j(\alpha^i)\right)^2.$$
		Using that $a_i \in \Q$, the factor $\sum_{i=1}^n a_i\sigma_j(\alpha^i)$ is zero if
		$$0 = \sigma_j(\sum_{i=1}^n a_i\alpha^i) = \sum_{i=1}^n a_i\alpha^i.$$
		Since the $\alpha^i$'s form a $\Q$-basis of $\mathbb{F}$, this expression is never zero, except when all $a_i$ are zero.
	\end{proof}
	\begin{rem}
		From the proof above, we know that all groups of the form $G_{\varphi_\F}$ are non-singular over $\Q$. However, they are not over $\F$ itself, as for suitable $(a_1 , \ldots , a_n) \in \F^n$, the expression $\sum_{i=1}^n a_i\sigma_j(\alpha^i)$ appearing in the determinant can be zero. Non-singularity is therefore not preserved by field extensions. Also, not all of these groups give rise to a non-singular Lie group $(G_{\varphi_\F})^\R$. For example, if we take $\F = \Q(i)$, then $(G_{\varphi_\F})^\R$ is non-singular, in contrast to the case $\F = \Q(\sqrt{2})$.
	\end{rem}
	Note that these examples also illustrate that the bound $\RF_{G_\varphi} \preceq \log^{d'}$ with
	$$d' =  \min\{\max_{j=1}^n\{\rank_\mathbb{Q}\sum_{i=1}^n a^{(j)}_iA_i\}\mid a^{(1)} \text{ to }a^{(n)}\text{ is a basis of }\mathbb{Q}^n\},$$
	i.e. the value $d(\varphi^\C)$ over the field $\Q$ instead of over $\C$, does not give a sharp upper bound.
	
	\section{Lower Bound} \label{sec_lower}
	In Theorem \ref{thm_upper_bound}, we have shown an upper bound for $\RF_{G_\varphi}$ of the form $\log^{d(\varphi^\C)+1}$. In this section, we will prove a lower bound of the form $\log^{\delta+1}$. By giving an alternative characterization of $d(\varphi^\C)$, we are able to establish a clear connection between $d(\varphi^\C)$ and $\delta$, given in Proposition \ref{prop_connection_upper_lower}. In fact, we believe both numbers always agree. Therefore, we conjecture that the bound of Theorem \ref{thm_upper_bound} is exact.
	\subsection{Proof of the Lower Bound} \label{sec_pr_lower}
	For the remainder of this article, we will fix the following notation:
	\begin{nota}
		Given matrices $A_1, A_2, \ldots , A_n \in \C^{m\times m}$. We denote the matrix $\sum_{i=1}^n x_iA_i$ in $\C[x_1, \dots, x_n]$ by $M_x$. If we evaluate the $x_i$ in concrete values $a_i \in \C$, we will write $M_a$ for $\sum_{i=1}^n a_iA_i \in \C^{n \times n}$.
	\end{nota}
	\begin{nota}
		For us, a $d\times d$ minor of a matrix $M$ is the determinant of a $d\times d$ submatrix of $M$. We will use the notation $I_d(M)$ to denote the ideal generated by all $d\times d$ minors of $M$. In case of $I_d(M_x)$, the corresponding variety (of common zeroes of the polynomials) will be denoted by $V(I_d(M_x))$.
	\end{nota}
	\begin{nota}
		The notation $I_d(M) = (q_1, \ldots , q_s)$ will be used, when we want to implicitly say that $q_1$ up to $q_s$ (for some $s\in \N$) are precisely the $d\times d$ minors of $M$.
	\end{nota}
	Throughout this subsection, we also fix the group $G_\varphi$, where the map $\varphi: \Z^m\times \Z^m \to \Z^n$ is defined by matrices $A_1, A_2, \ldots , A_n \in \Z^{m\times m}$. We wish to prove the following result: 
	\begin{thm}\label{thm_lower_bound}
		If there exist a non-trivial $v = (v_1, \ldots, v_n) \in \Z^n$ such that
		$$\left(\sum_{i=1}^n v_ix_i\right)^d \in I_d(M_x) \subset \Z[x_1, \ldots , x_n],$$
		then $\log^{d+1} \preceq \RF_{G_\varphi}$.
	\end{thm}
	Before we prove the above, we present some easy lemmas.
	\begin{lemma}
		If $\left(\sum_{i=1}^n v_ix_i\right)^d \in I_d(M_x) = (q_1, \ldots , q_s) \subset \Z[x_1, \ldots , x_n]$, then 
		$$\left(\sum_{i=1}^n v_ix_i\right)^d = \sum_{j=1}^s \lambda_j q_j$$
		with $\lambda_j \in \Z$.
	\end{lemma}
	The proof is immediate using the definition and the fact that $I_d(M_x)$ is generated by homogeneous polynomials of degree $d$.
	\begin{lemma}
		If there exist non-trivial rational numbers $v'_1$ to $v'_n$ in $\Q$ such that $\left(\sum_{i=1}^n v'_ix_i\right)^d$ lies in the ideal $I_d(M_x)\otimes_\Z \C$ over $\C[x_1 , \dots, x_n]$, then there exist non-trivial integers $v_i \in \Z$ such that
		$$\left(\sum_{i=1}^n v_ix_i\right)^d \in I_d(M_x) \subset \Z[x_1, \ldots , x_n].$$
	\end{lemma}
	\begin{proof}
		Let $m_1$ up to $m_l$ be an enumeration of all monomials in $\Z[x_1, \ldots , x_n]$ with coefficient $1$ of total degree $d$. On the one hand, we have
		$\left(\sum_{i=1}^n v'_ix_i\right)^d = \sum_{k=1}^l b_k(v_1', \ldots , v_n')m_k$,
		where $b_k$ is a monomial in $\Z[v_1', \ldots, v_n']$ of total degree $d$. On the other hand, if $I_d(M_x) = (q_1, \ldots , q_s)$, we can rewrite
		$\sum_{j=1}^s\lambda_jq_j = \sum_{k=1}^l a_k(\lambda_1, \ldots , \lambda_s)m_k$,
		where $a_k$ is a linear expression in $\Z[\lambda_1, \ldots , \lambda_s]$.
		
		By the previous lemma (over $\C$), the condition $\left(\sum_{i=1}^n v'_ix_i\right)^d\in I_d(M_x)\otimes_\Z \C$ implies the existence of $\lambda_j\in\C$ such that
		$\left(\sum_{i=1}^n v'_ix_i\right)^d = \sum_{j=1}^s \lambda_j q_j,$
		or equivalently
		$\sum_{k=1}^l b_k(v_1', \ldots , v_n')m_k = \sum_{k=1}^l a_k(\lambda_1, \ldots , \lambda_n)m_k.$
		Since this expression is an equality of polynomials in the $x_i$, it is equivalent to the following system of equations: 
		$$\forall 1\leq k \leq l: b_k(v_1', \ldots , v_n') = a_k(\lambda_1, \ldots , \lambda_n).$$
	These equations are linear in $\lambda_j$ with rational coefficients. Hence, a solution for $\lambda_j$ over $\C$ readily implies a solution for $\lambda_j$ over $\Q$. It is immediate that a solution over $\Q$ implies a solution over the integers.
	\end{proof}
	Since the lemma above tells us that Theorem \ref{thm_lower_bound} also applies when we replace $I_d(M_x)$ by $I_d(M_x)\otimes_\Z \C$. We will often write $I_d(M_x)$ when we mean $I_d(M_x)\otimes_\Z \C$, in particular in section \ref{sec_examples}.
	\begin{ex}
		Consider $H_3(\Z[i])$ as presented in Example \ref{ex_Heis}. In order to show that $\log^3 \preceq \RF_{H_3(\Z[i])}$, we need to show that there exists a non-trivial solution to 
		$$\left(v_1x_1+v_2x_2\right)^2 = v_1^2 x_1^2 + 2v_1v_2 x_1x_2 + v_2^2x_2^2 = \lambda_1 x_1^2 + \lambda_2 x_1x_2 + \lambda_3 x_2^2 + \lambda_4 (x_1^2+x_2^2),$$
		since $I_2(M_x) = \left(x_1^2,x_1x_2,x_2^2,x_1^2+x_2^2\right)$.
		This reduces to the following system
		$$\begin{cases}
			v_1^2 & = \lambda_1 + \lambda_4\\
			2v_1v_2 & = \lambda_2 \\
			v_2^2 & = \lambda_3 + \lambda_4.
		\end{cases}$$
	A possible solution that is non-trivial in $(v_1,v_2)$ is given by $v_1 = \lambda_1 = 1$ and $v_2 = \lambda_2 = \lambda_3 = \lambda_4 = 0$. Note that $\log^3 \preceq \RF_{H_3(\Z)}$ also follows from the fact that the discrete Heisenberg $H_3(\Z)$ embeds in $H_3(\Z[i])$.
	\end{ex}

	Now we proceed to prove the claimed lower bound.
	\begin{rem}
		Given an $m\times m$ matrix $M \in \Z^{m\times m}$. Recall that the Smith Normal Form of $M$ is a decomposition $M = P\Lambda Q$ with $P,Q \in \GL(m, \Z)$ and $\Lambda \in \Z^{m\times m}$ diagonal with entries $\mu_1$ to $\mu_m$. In the lemma below, we will use that $|\Im_{\Z_{p^k}}M|$ can be computed using the Smith Normal Form via
		$$\dfrac{|\Z_{p^k}^m|}{|\ker\{\Z_{p^k} \to \Z_{p^k}: x\mapsto Mx\}|} = \dfrac{p^{mk}}{|\ker\{\Z_{p^k} \to \Z_{p^k}: x\mapsto \Lambda x\}|} = \dfrac{p^{mk}}{\prod_{i=1}^m \gcd(\mu_i,p^k) }.$$
	\end{rem}
	\begin{lemma} \label{lem_lower_estimate}
		If there exist a non-trivial vector $v = (v_1, \dots , v_n) \in \Z^n$ such that 
		$$\left(\sum_{i=1}^n v_ix_i\right)^d \in I_d(M_x) \subset \Z[x_1, \ldots, x_n],$$
		then for $\kappa \in \N$ we have $D(0, \lcm(1,2, \dots , \kappa)v) > \kappa^{d+1}$.
	\end{lemma}
	\begin{proof}
		By Lemma \ref{lem_Dov}, $D(0, \lcm(1,2, \dots , \kappa)v)$ equals the following expression:
		$$\min\{|\Im_{\Z_{p^k}} M_a|\cdot p^k \mid \ggd(a) = 1, \lcm(1,2, \dots , \kappa)a^Tv \not\equiv 0 \mod p^k\}.$$
		Now take the power $p^k$ and the vector $a$ realizing this minimum. Let $a^Tv$ equal $bp^{f}$ with $\gcd(b,p) = 1$ and thus $f < k$. The condition  
		$$\lcm(1,2, \dots , \kappa)a^Tv \equiv \lcm(1,2, \dots , \kappa)bp^{f} \not\equiv 0 \mod p^k $$
		implies that $\lcm(1,2, \dots , \kappa) \not\equiv 0 \mod p^{k-f}$, and therefore, $\kappa < p^{k-f} \leq p^k$.
		
		By the assumption, $\left(\sum_{i=1}^n v_ix_i\right)^d$ is a $\Z$-linear combination of $(d\times d)$-minors of $M_x$. Hence, by evaluating in $a$, we observe that $(a^Tv)^d$ lies in the ideal $I_d(M_a) \subset \Z$ generated by the $(d\times d)$-minors of $M_a$. This ideal is spanned by the greatest common divisor of all $(d\times d)$-minors of $M_a$. Let $\mu_1$ to $\mu_m$ be the invariant factors of the Smith Normal Form of $M_a$. By its properties, we now know
		$$\prod_{i=1}^d\mu_i = \gcd( d\times d \text{ minors of }M_a) \mid (a^Tv)^d = (bp^{f})^d .$$
		In particular, we see that $\sum_{i=1}^d \log_p(\gcd(\mu_i,p^k)) \leq df$.
		We now obtain the following estimate
		\begin{equation*}
			\begin{split}
				|\Im_{\Z_{p^k}} M_a |\cdot p^k & = p^{k + km - \sum_{i=1}^m \log_p(\gcd(\mu_i,p^k))} \\
				& \geq p^{(k-f) + km - \sum_{i=1}^d \log_p(\gcd(\mu_i,p^k)) - \sum_{i=d+1}^m \log_p(\gcd(\mu_i,p^k))} \\
				& \geq p^{(k-f) + km - df - k(m-d)} \\
				& \geq p^{(d+1)(k-f)}\\
				& > \kappa^{d+1}.
			\end{split}
		\end{equation*}
		This ends the proof.
	\end{proof}
	\begin{proof}[Proof of Theorem \ref{thm_lower_bound}]
		The lower bound follows from Lemma \ref{lem_lower_estimate}. Indeed, using the standard generators of $\Z^m\times \Z^n$ as a set, we easily see that
		$$(0, \lcm(1,2, \dots, \kappa)v) \in B_{G_\varphi}(C_1\lcm(1,2, \dots , \kappa))$$
		for some $C_1$ depending on $v$.
		For this element we have the estimate
		$$D(0, \lcm(1,2, \dots , \kappa)v) > \kappa^{d+1}.$$
		By the Prime Number Theorem, we know that $\log(\lcm(1 , 2, \dots , \kappa)) \leq C_2 \kappa$ for $C_2>0$ and all $\kappa \in \mathbb{N}$. Therefore,
		$$D(0, \lcm(1,2, \dots , \kappa)v) \geq \dfrac{1}{C_2^{d+1}}\log^{d+1}(\lcm(1, 2 , \dots, \kappa)).$$
		This shows the result.
	\end{proof}
	
	\subsection{Connection With The Upper Bound}
	In this subsection, we relate the lower bound of Theorem \ref{thm_lower_bound} to the upper bound of Theorem \ref{thm_upper_bound}. In order to do so, we first need to give an equivalent definition of $d(\varphi^\C)$. 
	
	\begin{lemma}
		The following numbers are equal:
		\begin{itemize}
			\item $d(\varphi^\C) = \min\{\max_{j=1}^n\{\rank_\mathbb{C}M_{a^{(j)}}\}\mid a^{(1)} \text{ to }a^{(n)}\text{ is a basis of }\mathbb{C}^n\},$
			\item $d_2(\varphi^\C) = \max\{\min\{\rank_{\mathbb{C}} M_a\mid a\in \mathbb{C}^n, a^Tv \neq 0\} \mid 0\neq v\in \mathbb{C}^n\}.$
		\end{itemize}
	\end{lemma}
	\begin{proof}
		First, we focus on the inequality $d(\varphi^\C) \leq d_2(\varphi^\C)$.
		
		Take $v_1$ realizing the maximum of $d_2(\varphi^\C)$ and its corresponding $a^{(1)}$ realizing the minimum. Since the matrix given by the row $(a^{(1)})^T$ is not of full rank, we find a non-zero vector $v_2$ such that $(a^{(1)})^Tv_2 = 0$. Now take $a^{(2)}$ realizing the minimum
		$$\min\{\rank_{\mathbb{C}} M_a\mid a\in \mathbb{C}^n, a^Tv_2 \neq 0\}.$$
		Now the matrix
		$$A_2 = \begin{pmatrix} (a^{(1)})^T \\ (a^{(2)})^T \end{pmatrix}$$
		is still not of full rank, so we can find a non-zero vector $v_3$ such that $A_2v_3 = 0$. Let $a^{(3)}$ be the projection realizing the minimum for $v_3$, as we also did for $v_2$ and $a^{(2)}$.
		
		Proceed like this until one has $n$ vectors $v_1$ to $v_n$ and corresponding projections $a^{(1)}$ to $a^{(n)}$. Those projections are a basis. Indeed, suppose
		$$ \sum_{j=1}^n \lambda_j a^{(j)} = 0,$$ then $\lambda_n = 0$, because
		$$0= 0^Tv_n = \left(\sum_{j=1}^n \lambda_j a^{(j)}\right)^Tv_n = \lambda_n ((a^{(n)})^T v_n),$$
		and $(a^{(n)})^T v_n \neq 0$ by construction. Continuing in this fashion with $v_{n-1}$ up to $v_1$ shows that all $\lambda$ are zero.
		
		Since $v_1$ and $a^{(1)}$ realize $d_2(\varphi^\C)$, we know that for all $1 \leq j \leq n$
		$$\rank_\mathbb{C} M_{a^{(j)}} \leq d_2(\varphi^\C),$$
		and that $\max_{j=1}^n\{\rank_\mathbb{C} M_{a^{(j)}}\} = d_2(\varphi^\C)$.
		In particular, this shows that $d(\varphi^\C) \leq d_2(\varphi^\C)$.
		
For the converse inequality $d_2(\varphi^\C) \leq d(\varphi^\C)$, take a basis $a^{(1)}$ to $a^{(n)}$ realizing the minimum. Take any vector $0\neq v\in \mathbb{C}^n$. There must exist a $1\leq j \leq n$ such that $(a^{(j)})^Tv \neq 0$. Now
		$$ \min\{\rank_{\mathbb{C}}  M_a\mid a\in \mathbb{C}^n, a^Tv \neq 0\} \leq \rank_{\mathbb{C}} M_{a^{(j)}} \leq d(\varphi^\C).$$
		Since $v$ was arbitrary, this shows that $d_2(\varphi^\C) \leq d(\varphi^\C)$.
	\end{proof}
	\begin{lemma} \label{lem_d2_d3}
		There exists an integral vector $v\in \Z^n$ realizing the maximum of 
		$$d_2(\varphi^\C) = \max\{\min\{\rank_{\mathbb{C}} M_a\mid a\in \mathbb{C}^n, a^Tv \neq 0\} \mid 0\neq v\in \mathbb{C}^n\},$$
		i.e. $d_2(\varphi^\C)$ equals
		$$d_3(\varphi^\C) = \max\{\min\{\rank_{\mathbb{C}} M_a\mid a\in \mathbb{C}^n, a^Tv \neq 0\} \mid 0\neq v\in \Z^n\}.$$
	\end{lemma}
	\begin{proof}
		It is easily seen that $d_3(\varphi^\C) \leq d_2(\varphi^\C)$. We will show that $d_2(\varphi^\C) \leq d_3(\varphi^\C)$.
		
		First we rewrite the expressions for $d_2(\varphi^\C)$ and $d_3(\varphi^\C)$. Specifically, we claim $\C$ may be replaced by some large enough number field $\mathbb{L}$ in both definitions. Let $v_\Z$ be the vector in $\Z^n$ realizing the maximum of $d_3(\varphi^\C)$. The minimum
		$d_3(\varphi^\C) = \min\{\rank_{\mathbb{C}} M_a\mid a\in \mathbb{C}^n, a^Tv_\Z \neq 0\}$ can be expressed as
		$$(\forall a: a^Tv_\Z \neq 0 \Rightarrow \rank M_a \geq d_3(\varphi^\C))\wedge (\exists \bar{a}: (\bar{a}^Tv_\Z\neq 0) \wedge \rank M_a = d_3(\varphi^\C)).$$
		Note that inequalities concerning the rank of $M_a$ can be expressed using polynomials. Indeed, the rank is at least $d_3(\varphi^\C)$ if some minor of size at least $d_3(\varphi^\C)$ is non-zero. It is equal to $d_3(\varphi^\C)$, if furthermore all minors of size strictly larger than $d_3(\varphi^\C)$ are zero. Hence, this is a sentence in the language of rings, satisfied in the algebraically closed field $\mathbb{C}$. By the Lefschetz Principle (\cite[Theorem 3.5.5]{hils2019first}), it must therefore be satisfied in any algebraically closed field of characteristic zero, in particular in the algebraic closure of $\mathbb{Q}$, which we denote by $\bar{\mathbb{Q}}$. We can therefore conclude that
		$$d_3(\varphi^\C) = \min\{\rank_{\bar{\mathbb{Q}}} M_a\mid a\in \bar{\mathbb{Q}}^n, a^Tv_\Z \neq 0\}.$$
		The vector $\bar{a}$ realizing the minimal rank of $M_{\bar{a}}$ from the statement now lives in $\bar{\mathbb{Q}}^m$. Hence, it lies in some number field $\mathbb{F}_{d_3(\varphi^\C)}$. Hence, $d_3(\varphi^\C)$ also equals
		$$\max\{\min\{\rank_{\mathbb{L}} M_a\mid a\in \mathbb{L}^n, a^Tv \neq 0\} \mid 0\neq v\in \Z^n\}$$
		for all extensions $\mathbb{L}$ of $\mathbb{F}_{d_3(\varphi^\C)}$. 
		
		Analogously, we can express $d_2(\varphi^\C)$ by 
		$$\exists v \neq 0: (\forall a: a^Tv \neq 0 \Rightarrow \rank M_a \geq d_2)\wedge (\exists \bar{a}: (\bar{a}^Tv\neq 0) \wedge \rank M_a = d_2),$$
		saying that there exists a vector $v$ realizing the maximum, so for which the minimum of the rank of $M_a$ is precisely $d_2(\varphi^\C)$. Hence, we can find a number field $\mathbb{F}_{d_2(\varphi^\C)}$ such that $v$ and its corresponding $a$ realizing $d_2(\varphi^\C)$ lie in this field. Therefore,
		$d_2(\varphi^\C)$ equals
		$$\max\{\min\{\rank_{\mathbb{L}} M_a\mid a\in \mathbb{L}^n, a^Tv \neq 0\} \mid 0\neq v\in \mathbb{L}^n\}$$
		for all extensions $\mathbb{L}$ of $\mathbb{F}_{d_2(\varphi^\C)}$. 
		
		Take the number field $\mathbb{L}$ to be an extension of both $\mathbb{F}_{d_2(\varphi^\C)}$ and $\mathbb{F}_{d_3(\varphi^\C)}$. Without loss of generality, we may assume $\mathbb{L}$ is Galois over $\mathbb{Q}$. Let $\{\sigma_1, \dots \sigma_k\}$ be its automorphisms.
		
		Now consider $v\in \mathbb{L}^n$ realizing $d_2(\varphi^\C)$. We will construct $v_\Z$.
		
		We know that $\rank_\mathbb{L} M_a$ is at least $d_2(\varphi^\C)$ for all $a\in\mathbb{L}^n$ with $a^Tv \neq 0$. Replace $v$ with a non-zero multiple such that $v\in (\mathcal{O}_\mathbb{L})^n$, i.e. such that $v$ lies over the algebraic integers in $\mathbb{L}$. Now take $\lambda\in\mathcal{O}_\mathbb{L}$ such that
		$v_\Z = \sum_{j=1}^k \sigma_j(\lambda v)$
		is non-zero. (Here, $\sigma_i$ is applied to the entries of the vector. The existence of such a vector $v_\Z$ follows from \cite[Corollary 3.1.2]{Winter1974Fields}.) By construction, $v_\Z$ is a vector in $\Z^n$. 
		
		We claim that if $a\in\mathbb{L}^n$ and $a^Tv_\Z \neq 0$, then the rank of $M_a$ is at least $d_2(\varphi^\C)$. This shows that
		$$d_2(\varphi^\C) \leq  \max\{\min\{\rank_{\mathbb{L}} M_a\mid a\in \mathbb{L}^n, a^Tv \neq 0\} \mid 0\neq v\in \Z^n\} = d_3(\varphi^\C),$$
		where the last equality was shown in the first part of the proof. 
		
		Suppose $a^Tv_\Z \neq 0$. This equals $\sum_{j=1}^k a^T\sigma_j(\lambda v)$, so at least one term has to be non-zero. Say $a^T\sigma(\lambda v) \neq 0$. Then $\sigma^{-1}(a)^Tv \neq 0$. Since this is non-zero and $v$ realizes the maximum for $d_2(\varphi^\C)$, we know that
		$$d_2(\varphi^\C) \leq \rank_\mathbb{L} M_{\sigma^{-1}(a)}.$$
		However,
		$$M_{\sigma^{-1}(a)} = \sum_{i=1}^n \sigma^{-1}(a_i)A_i = \sigma^{-1}(\sum_{i=1}^na_iA_i) = \sigma^{-1}(M_a),$$
		using that the matrices $A_i$ are integral. In total, we find
		$$d_2(\varphi^\C) \leq \rank_\mathbb{L} M_{\sigma^{-1}(a)} = \rank_\mathbb{L} \sigma^{-1}(M_{a}) = \rank_\mathbb{L} M_{a}.$$
		This shows the claim and ends the proof of the equality $d_2(\varphi^\C) = d_3(\varphi^\C)$.
	\end{proof}
	We can reinterpret $d_2(\varphi^\C)$ (and $d_3(\varphi)$) using the ideals $I_d(M_x)$.
	\begin{prop} \label{prop_reform_hyperplane}
		We have the equalities
		\begin{equation*}
			\begin{split}
				d_2(\varphi^\C) 
				& = \max\{d \mid V(I_d(M_x)) \text{ lies in a (non-zero) linear hyperplane}\}\\
				& = \max\{d \mid \exists v \in \C^n: \exists k\geq d: \left(v^Tx\right)^k \in I_d(M_x)\}\\
				& = \max\{d \mid \exists v \in \Z^n: \exists k\geq d: \left(v^Tx\right)^k \in I_d(M_x)\}.
			\end{split}
		\end{equation*}
	\end{prop}
	\begin{proof}
		We start by showing the first equality.
		
		The variety $V(I_{d_2(\varphi^\C)}(M_x))$ lies in a linear hyperplane. Indeed, take $v\in \C^n$ to be the vector realizing the maximum in the definition of $d_2(\varphi^\C)$. In other words, we take the vector such that $d_2(\varphi^\C) = \min\{\rank_{\mathbb{C}} M_a\mid a\in \mathbb{C}^n, a^Tv \neq 0\}$. Hence, $\rank_{\mathbb{C}} M_a < d_2(\varphi^\C)$ only if $a^Tv = 0$. Recall that in general the rank of a matrix $M\in \C^{m\times m}$ is smaller than a number $d$ if and only if all $(d\times d)$-minors of $M$ are zero. Therefore, $V(I_{d_2(\varphi^\C)}(M_x))$ lies in the hyperplane $v^Tx=0$.
		
		Now let $d > d_2(\varphi^\C)$. We claim that $V(I_d(M_x))$ does not lie in a linear hyperplane. Indeed, if it would lie in a hyperplane $w^Tx=0$, then the rank of $M_a$ would be at least $d$ for all $a\in \C^n$ with $a^Tw \neq 0$. Therefore, $d \leq \min\{\rank_{\mathbb{C}} M_a\mid a\in \mathbb{C}^n, a^Tw \neq 0\} \leq d_2(\varphi^\C)$. This is a contradiction.
		
		The second equality in the statement is precisely the content of Hilbert's Nullstellensatz, in combination with the fact that $k\geq d$, since $I_d(M_x)$ is generated by homogeneous polynomials of degree $d$. The third equality is a reformulation of Lemma \ref{lem_d2_d3}: it says that the vector $v$ attaining the maximum can be chosen to lie in $\Z^n$.
	\end{proof}
	Combining Proposition \ref{prop_reform_hyperplane} above and Theorems \ref{thm_upper_bound} and \ref{thm_lower_bound}, we can formulate the following link between the upper and lower bound.
	\begin{prop} \label{prop_connection_upper_lower}
		Let $G_\varphi$ be a two-step $\mathcal{I}$-group. Then, $\log^{\delta +1} \preceq \RF_{G_\varphi} \preceq \log^{d(\varphi^\C)+1}$, where
		$$\delta = \max\{d \mid \exists v \in \Z^n: \left(v^Tx\right)^d \in I_d(M_x)\} \leq  \max\{d \mid \exists v \in \Z^n: \exists k\geq d: \left(v^Tx\right)^k \in I_d(M_x)\} = d(\varphi^\C).$$ 
	\end{prop}

	The authors have no knowledge of an example where $\delta$ does not equal $d(\varphi^\C)$. Hence, we state the following conjecture:
	\begin{conj} \label{conj_minors}
		Let $G_\varphi$ be a two-step $\mathcal{I}$-group. Then $\RF_{G_\varphi} = \log^{d(\varphi^\C)+1}$.
	\end{conj}
	In particular, we conjecture that the power of logarithm is always odd.
	\section{$2$-step nilpotent groups with commutator of small rank} \label{sec_examples}
	In this section, we will show that Conjecture \ref{conj_minors} holds at least for $\mathcal{I}(m,1)$- and $\mathcal{I}(m,2)$-groups.
	
	As a first observation, we show that in general we may assume the hyperplane $\sum_{i=1}^n v_ix_i = 0$ of Proposition \ref{prop_connection_upper_lower} is given by $x_n = 0$.
	\begin{lemma} \label{lem_choose_hyperplane}
		Let $d\in \N$. Let $M$ be a skew-symmetric, $m\times m$ matrix of the form $\sum_{i=1}^n x_iA_i$ with $A_i \in \C^{m\times m}$, for which $V(I_d(M))$ lies in a (complex) hyperplane $\sum_{i=1}^n v_ix_i = 0$. Then there exists a skew-symmetric, $m\times m$ matrix $M'$ of the form $\sum_{i=1}^n x'_iA'_i$ with $A'_i \in \C^{m\times m}$ such that
		$$\left(\sum_{i=1}^n v_ix_i\right)^d \in I_d(M) \Leftrightarrow \left(x_n'\right)^d \in I_d(M').$$
	\end{lemma}
	\begin{proof}
		The vector $v = (v_1, \ldots , v_n)$ is non-zero, since $\sum_{i=1}^n v_ix_i = 0$ represents a non-zero hyperplane. Hence, we can take a transformation $P \in \GL(n, \C)$ such that $Pv = e_n$.
		
		Define the vector of variables $x' = (x_1', \ldots, x_n')$ via $P^Tx' = x$ with $x = (x_1, \ldots , x_n)$. Define also $A'_j$ by $\sum_{i=1}^n P_{j,i}A_i$. Now, we see that
		$$M = \sum_{i=1}^n x_iA_i =  \sum_{i=1}^n(\sum_{j=1}^n P_{j,i}x'_j)A_i = \sum_{j=1}^n x'_j\sum_{i=1}^n P_{j,i}A_i = \sum_{j=1}^n x'_jA'_j = M',$$
		and
		$$x_n' = e_n^Tx' = (Pv)^Tx' = v^Tx.$$
		Therefore, $I_d(M) = I_d(M')$ under the correspondence $P^Tx' = x$, where the linear combination $v^Tx$ defining the hyperplane $v^Tx = 0$ becomes the monomial $x_n'$ with corresponding hyperplane $x_n'=0$.
	\end{proof}
	
	Now, we will show Conjecture \ref{conj_minors} holds for $\mathcal{I}(m,1)$-groups.
		\begin{lemma} \label{lem_one_matrix}
		Let $G_\varphi$ be an $\mathcal{I}(m,1)$-group. Then $\varphi$ is given by $\varphi(w,w') = w^TAw'$ for some skew-symmetric matrix, and $\RF_{G_\varphi}(r) = \log^{(\rank A) + 1}(r)$.
	\end{lemma}
	\begin{proof}
		By definition, $\varphi$ is given by $\varphi(w,w') = w^TAw'$ for some skew-symmetric matrix. Also, it is clear that $d(\varphi^\C) = \rank A$. Note that $A$ has a non-zero (integral) $(\rank A \times \rank A)$-minor $\Delta$ by definition of the rank. Now $\Delta x^{\rank A}$ is a $(\rank A \times \rank A)$-minor of $M_x = xA$. Therefore, Theorem \ref{thm_lower_bound} says that $\RF_{G_\varphi} = \log^{(\rank A) + 1}$. 
	\end{proof}
	\begin{cor}
		For every odd number $2d+1 \geq 3$, there exists a two-step $\mathcal{I}$-group $G$ such that $\RF_G = \log^{2d+1}$.
	\end{cor}
	
	From now on, we will focus on matrices $M_x$ of the form $x_1A_1 + x_2A_2$ associated to $\varphi = (A_1,A_2)$ of an $\mathcal{I}(m,2)$-group $G_\varphi$. We will prove the following statement.
	\begin{prop} \label{prop_conj_2dim}
		Let $M_x$ be a skew-symmetric, $m\times m$ matrix of the form $x_1A_1 + x_2A_2$ with $A_i \in \C^{m\times m}$. If $d$ is the largest number such that the variety $V(I_d(M_x))$ lies in the hyperplane $x_2 = 0$, then $(x_2)^{d} \in I_d(M_x)$.
	\end{prop}
	By Lemma \ref{lem_choose_hyperplane} and Proposition \ref{prop_connection_upper_lower}, this result implies the following:
	\begin{thm} \label{thm_Im2_groups}
		Let $G_\varphi$ be an $\mathcal{I}(m,2)$-group. Then $\RF_{G_\varphi} = \log^{d(\varphi^\C)+1}$.
	\end{thm}
	\begin{proof}
		Denote $x_1A_1 + x_2A_2$ by $M_x$, where $\varphi = (A_1,A_2)$. By Proposition \ref{prop_connection_upper_lower}, $V(I_{d(\varphi^\C)}(M_x))$ lies in a hyperplane $v_1x_1 + v_2x_2 = 0$ with $0\neq (v_1,v_2) \in \Z^2$ (and $V(I_d(M_x))$ does not lie in a hyperplane for $d>d(\varphi^\C)$). Using Lemma \ref{lem_choose_hyperplane}, we find a matrix $M'$ of the form $x_1'A_1'+x_2'A_2'$ such that $V(I_d(M'))$ lies in the hyperplane $x_2' = 0$. By Proposition \ref{prop_conj_2dim}, we know this implies that $(x_2')^d \in I_d(M')$ and therefore $(v_1x_1 + v_2x_2)^d \in I_d(M)$ by Lemma \ref{lem_choose_hyperplane}. Now Proposition \ref{prop_connection_upper_lower} (or Theorem \ref{thm_upper_bound} and Theorem \ref{thm_lower_bound}) give the desired result.
	\end{proof}
	
	Now, we will proceed to prove Proposition \ref{prop_conj_2dim}. Consider therefore a matrix $M$ of the form $xA_1 +yA_2$ with $A_1$ and $A_2$ skew-symmetric, complex $m\times m$-matrices. The proof consists of the following steps:
	\begin{enumerate}
		\item It is known that $I_d(M) = I_d(PMQ)$ for all $P,Q \in \GL(m, \C)$ (see for example \cite[Chapter 1, Theorem 2]{northcott1976finite}). Hence, we may replace $xA_1 + yA_2$ by a well-chosen congruent matrix $Q(xA_1+yA_2)Q^T$ with $Q\in \GL(m, \C)$, see Theorem \ref{thm_pencils}.
		\item Then, we calculate the largest $d\in \N$ such that $V(I_d(M))$ lies in the hyperplane $y=0$. This is done in Lemma \ref{lem_dphi_pencil}.
		\item Finally, we show that $y^d \in I_d(M)$ for this $d\in\N$. This is the content of Proposition \ref{prop_dy_in_ideal}.
	\end{enumerate}
	
	We start by introducing this normal form for $xA_1 + yA_2$ and the corresponding terminology. 
	\begin{df}
		A matrix of the form $xA_1 + yA_2$ with skew-symmetric $A_1, A_2 \in \mathbb{C}^{m\times m}$ is called a \emph{skew matrix pencil}.
	\end{df}
	\begin{df} \label{df_blocks}
		Let $k \in \N$, and let $\alpha \in \mathbb{C}$. We define the following matrices:
\begin{equation*}
	\begin{split}
		F(\infty, k) & = \begin{pmatrix}0 & \mathcal{E}_k(y,x) \\ -\mathcal{E}_k(y,x) & 0\end{pmatrix} \in \C^{2k\times 2k}, \\
		F(\alpha, k) & = \begin{pmatrix}0 & \mathcal{E}_k(x-\alpha y,y) \\ -\mathcal{E}_k(x-\alpha y,y) & 0\end{pmatrix} \in \C^{2k\times 2k}, \\
		S(k) & = \begin{pmatrix}0 & \mathcal{L}_k(x,y) \\ -(\mathcal{L}_k(x,y))^T & 0\end{pmatrix} \in \C^{(2k+1)\times (2k+1)}.
	\end{split}
	\end{equation*}
	Here, the matrices $\mathcal{E}_k(y,x)$, $\mathcal{E}_k(x-\alpha y,y)$ and $\mathcal{L}_k(x,y)$ are given by the following formulas:
	\begin{equation*}
		\mathcal{E}_k(a,b) = \begin{pmatrix}&&&&a\\&&&a&b\\&&\iddots&b&\\&a&\iddots&&\\a&b&&&\end{pmatrix} \in \mathbb{C}^{k\times k}, \quad \mathcal{L}_k(a,b) = \begin{pmatrix}a&0&\dots&0\\b&a&&0\\&\ddots&\ddots&\\0&&b&a\\0&\dots&0&b\end{pmatrix} \in \mathbb{C}^{(k+1)\times k}.
	\end{equation*}
	\end{df}
	By \cite{gauger1973classification,MR1683874}, we have the following result:
	\begin{thm} \label{thm_pencils}
		Let $xA_1 + yA_2 \in (\C[x,y])^{m\times m}$ be a skew matrix pencil. There exists a matrix $Q\in \GL(m, \C)$ such that $Q^T(xA_1 + yA_2)Q$ equals
		$$\left(\bigoplus_{i=1}^l \left( \bigoplus_{j=1}^{s_i} F(\alpha_i, k^{(i)}_j) \right) \right) \oplus \left( \bigoplus_{j=1}^{s_\infty} F(\infty, k^{(\infty)}_j) \right) \oplus \left( \bigoplus_{j=1}^{s_S} S(k^{(S)}_j) \right) \oplus 0,$$
		where $\alpha_1$ to $\alpha_l$ are distinct complex numbers, and $s_i, s_\infty, s_S \in \mathbb{N}$.
	\end{thm}
	The notation $A\oplus B$ is used for the block diagonal matrix
	$$ A\oplus B := \begin{pmatrix}A&0\\0&B\end{pmatrix}.$$
	The result therefore says that up to congruence the skew matrix pencil can be written as a block diagonal matrix with blocks as in Definition \ref{df_blocks}. If $\alpha_i$ appears in one of the blocks, then we denote $s_i$ for the number of blocks with this number. The numbers $k^{(i)}_j$ represent the sizes of those blocks. The size will be $(2k^{(i)}_j)\times (2k^{(i)}_j)$.
	
	Now, given the normal form, we will proceed to calculate the largest $d\in \N$ such that $V(I_d(M))$ lies in the hyperplane $y=0$. For this, we will need information about how the ideals generated by minors look like for matrices $\mathcal{E}_k(y,x)$, $\mathcal{E}_k(x-\alpha y,y)$ (with $\alpha \in \C$) and $S(k)$.
	\begin{nota}
		Denote the ideal generated by the homogeneous polynomials of degree $k$ in two variables $x$ and $y$ by $H_k(x,y)$, i.e. $H_k(x,y) = \left(x^i y^{k-i}\mid 0\leq i \leq k\right)$. 
	\end{nota}
	\begin{lemma} \label{lem_pencil_minors}
		The following holds:
		\begin{equation*}
			\begin{split}
			I_{k-1}(\mathcal{E}_k(y,x)) & = H_{k-1}(x,y), \\
			I_{k}(\mathcal{E}_k(y,x)) & = \left( y^k \right), \\
			I_{k-1}(\mathcal{E}_k(x-\alpha y,y)) & = H_{k-1}(x,y), \\
			I_{k}(\mathcal{E}_k(x-\alpha y,y)) & = \left( (x-\alpha y)^k \right), \\
			I_{2k}(S(k)) & = H_{2k}(x,y), \\
			I_{2k+1}(S(k)) & = \left( 0 \right).
			\end{split}
		\end{equation*}
	\end{lemma}
	\begin{proof}
		This proof follows from straightforward calculations. 
	\end{proof}
	\begin{rem}
		Suppose $tA_1 + A_2$ with $t\in \C$ is given. The rank of this matrix is at least $d$ if and only if there is a non-zero $d\times d$-minor. Consider $I_d(tA_1+A_2)$ this is an ideal over the PID $\C[t]$, hence $I_d(tA_1+A_2) = \left(q(t)\right)$ for some polynomial $q(t)$. If $tA_1 + A_2$ has rank at least $d$ for all $t\in \C$, then $q(t)$ can have no roots, i.e. $I_d(tA_1+A_2) = \left(1\right)$. The converse also holds. We have
		\begin{equation} \label{eq_tA1A_2}
		\forall t\in \C: \rank(tA_1+A_2) \geq d \Leftrightarrow I_d(tA_1+A_2) = \left(1\right).
		\end{equation}
		
		Now suppose $xA_1 + yA_2$ is given with $y\neq 0$, then we can set $t= x/y$ and $\rank(xA_1 + yA_2) = \rank((1/y)(xA_1+yA_2)) = \rank(tA_1+A_2)$. Therefore, Equation \eqref{eq_tA1A_2} implies 
		\begin{equation} \label{eq_xA1yA_2}
			\begin{split}
			\forall (x,y)\in \C\times\C_0: \rank(xA_1+yA_2) \geq d & \Leftrightarrow \exists l\in \N^\ast: I_d(xA_1+yA_2) = \left(y^l\right) \\ & \Leftrightarrow V(I_d(xA_1+yA_2)) \subset \{(x,y)\in \C^2\mid y = 0\}.
			\end{split}
		\end{equation}
	\end{rem}
	The condition of Proposition \ref{prop_conj_2dim} requires us to find the largest $d\in \N$ such that Equation \eqref{eq_xA1yA_2} holds. Equivalently, by Equation \eqref{eq_tA1A_2}, we need to find the largest $d\in\N$ such that $I_d(tA_1+A_2) = \left(1\right)$. For this, we will suppose that $xA_1+yA_2$ has the normal form of Theorem \ref{thm_pencils}, and we will use Lemma \ref{lem_pencil_minors} above to facilitate the calculations. The result is given in Lemma \ref{lem_dphi_pencil}, and is based on \cite[Lemma 2.4.7]{haslinger2001families}. To prove this lemma, we first need the following result, which is of a combinatorial nature.
	\begin{lemma} \label{lem_distinct_tuples}
		Let $V$ be a set consisting of $n$ distinct elements. Suppose they are divided in $l$ different subsets, named $\beta_1$ to $\beta_l$, with sizes $1\leq n_1 \leq n_2 \leq \ldots \leq n_{l}$ respectively. If $d \leq n - n_{l}$, then we can construct $2^d$ ordered tuples of length $d$, such that
		\begin{enumerate}
			\item each tuple consists of $d$ distinct elements,
			\item and for every tuple $T$ and every position $1\leq m \leq d$, we can find another tuple $T'$ such that the first $m-1$ entries agree and the $m$'th entry of $T'$ is an element of another subset.
		\end{enumerate}
	\end{lemma}
	\begin{proof}
		We will first endow $V$ with an ordering. For this, give an order to the $n_i$ elements of subset $\beta_i$. In this way, we can represent all $n$ elements by a vector $(i,j)$, where $i$ is the number of the subset and $j$ is the number of the element in the group ($1\leq i \leq l$ and $1\leq j \leq n_i$). Now use the lexicographic order with $(i,j) \leq (i',j')$ if $i < i'$ or $i = i'$ and $j \leq j'$, i.e.
		$$(1,1) < \ldots < (1,n_1) < (2,1) < \ldots < (2, n_2) < (3,1) < \ldots $$
		
		We construct a (directed) binary rooted tree with labels of depth $d$. The root is unlabeled. Every branch node has two children: the first child is the smallest element that does not appear as one of its ancestors, and the second child is the smallest element that does not appear as one of its ancestors and that comes from a different subset than the first child. Now, every tuple corresponds to a path from the root of the tree to a leaf, see Figure \ref{fig_tuples} for an example. 

	The existence of the tree is guaranteed by the following observation: Suppose we fixed the first $k<d\leq n - n_l$ entries of the tuples/labels of the tree.  For the first child, we still have elements left to choose from, since $k$ does not exceed the total number of elements. For the second child, we cannot choose the $k$ preceding elements and the elements of the subset of the first child. Together, that gives at most $k+ n_l < n$ forbidden elements. Hence, there is at least one element left to pick.  
	\end{proof}
	\begin{figure} \label{fig_tuples}
		\begin{center}
		\begin{tikzpicture}[level distance=1cm,
			level 1/.style={sibling distance=4cm},
			level 2/.style={sibling distance=2cm},
			level 3/.style={sibling distance=1cm}]
			\node {$\ast$}[grow'=down]
			child {node {$(2,1)$}
				child {node {$(2,2)$}
					child {node {$(3,1)$}}
					child {node {$(1,1)$}}}
				child {node {$(1,1)$}
					child {node {$(3,1)$}}
					child {node {$(2,2)$}}}}
			child {node {$(1,1)$}
				child {node {$(3,1)$}
					child {node {$(3,2)$}}
					child {node {$(2,1)$}}}
				child {node {$(2,1)$}
					child {node {$(3,1)$}}
					child {node {$(2,2)$}}}};
		\end{tikzpicture}
	\end{center}
		\caption{Example illustrating the construction of the tuples of Lemma \ref{lem_distinct_tuples}. There is one subset of one element and two subsets with two elements, hence the elements can be represented as the set $\{(1,1); (2,1); (2,2); (3,1); (3,2)\}$, yielding the tree above. Now, the first tuple of this tree is $\{(1,1); (2,1); (2,2)\}$. It has the first two entries in common with the second tuple $\{(1,1); (2,1); (3,1)\}$, but the last entry comes from a different subset.}
	\end{figure}
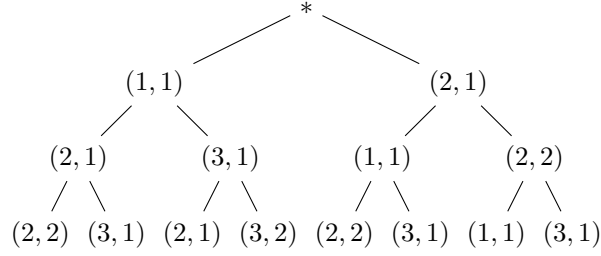
	\begin{df} \label{df_dy}
		Let $M = xA_1+yA_2$ be a skew matrix pencil of the form
		$$\left(\bigoplus_{i=1}^l \left( \bigoplus_{j=1}^{s_i} F(\alpha_i, k^{(i)}_j) \right) \right) \oplus \left( \bigoplus_{j=1}^{s_\infty} F(\infty, k^{(\infty)}_j) \right) \oplus \left( \bigoplus_{j=1}^{s_S} S(k^{(S)}_j) \right) \oplus 0,$$
		where $\alpha_1$ to $\alpha_l$ are distinct complex numbers, and $s_i, s_\infty, s_S \in \mathbb{N}$. Define the following number:
		$$d_y(xA_1+yA_2) := 2\left(f + \sum_{i=1}^l s_i - \max\{s_i\mid 1\leq i \leq l\}\right)$$
		with 
		$$f = \left(\sum_{i=1}^l \left( \sum_{j=1}^{s_i} (k^{(i)}_j - 1) \right) \right) + \left( \sum_{j=1}^{s_\infty} k^{(\infty)}_j \right) + \left( \sum_{j=1}^{s_S} k^{(S)}_j \right).$$
	\end{df}
	\begin{lemma} \label{lem_dphi_pencil}
		Given the situation of Definition \ref{df_dy}: If $y\neq 0$, then for all $x\in \mathbb{C}$, the rank of $M$ cannot be smaller than $d_y(xA_1+yA_2)$. In other words, the largest number $d\in \N$ such that $V(I_d(M))$ lies in the hyperplane $y=0$ equals $d_y(xA_1+yA_2)$.
	\end{lemma}
	\begin{proof}
		Since $y$ is not zero, we can divide the pencil $M = xA_1 + yA_2$ by $y$ and set $t := x/y$ to obtain $M' = tA_1 + A_2$. In terms of the sum decomposition of $M$, this runs down to replacing $x$ by $t$ and $y$ by $1$ in all summands provided in Definition \ref{df_blocks}. A summand coming from $F(\alpha,k)$ will be denoted by $F'(\alpha,k)$ and so on. Now, we can study $M'$ in the PID $\C[t]$. By Equations \eqref{eq_tA1A_2}-\eqref{eq_xA1yA_2}, the number $d_y(xA_1+yA_2)$ is the largest number $d$ such that $I_d(M') = \left(1\right)$.
		
		 Now, we will proceed to reduce each summand of $M'$ to its Smith Normal Form over $\C[t]$. Suppose first a summand $S'(k)$ is given. By Lemma \ref{lem_pencil_minors}, we know that $I_{2k}(S(k)) = H_{2k}(x,y)$ and $I_{2k+1}(S(k)) = \left(0\right)$. Therefore, $I_{2k}(S'(k)) = \left(1\right)$ and $I_{2k+1}(S'(k)) = \left(0\right)$. Its Smith Normal Form is thus given by $\mathbb{1}_{2k} \oplus 0$.
		 
		 Suppose now a summand $F'$ coming from $F(\alpha,k)$ or $F(\infty,k)$ is given. This summand is of the form 
		 $$\begin{pmatrix}0& \mathcal{E}_k'\\ -\mathcal{E}_k'&0\end{pmatrix}.$$
		 Obviously, this matrix has the same Smith Normal Form as
		 $$\begin{pmatrix}\mathcal{E}_k'&0\\ 0&\mathcal{E}_k'\end{pmatrix}.$$
		 The Smith Normal Form of $\mathcal{E}'_k$ can be easily found. Indeed, similar to the arguments from the previous case, by Lemma \ref{lem_pencil_minors}, we know that $I_{k-1}(\mathcal{E}'_k) = \left(1\right)$, so its $k-1$ first invariant factors are $1$. Its last invariant factor depends on the type of summand: $(t-\alpha )^k$ for $F'(\alpha, k)$ and $1$ for $F'(\infty,k)$.
		 
		 Above, we have argued that that matrix $M'$ is similar over $\C[t]$ to a diagonal matrix with the following number of ones:
		 $$2f = \left(\sum_{i=1}^l \left( \sum_{j=1}^{s_i} 2(k^{(i)}_j - 1) \right) \right) + \left( \sum_{j=1}^{s_\infty} 2k^{(\infty)}_j \right) + \left( \sum_{j=1}^{s_S} 2k^{(S)}_j \right).$$
		 The number $\sum_{i=1}^l 2s_i$
		 represents the number of diagonal elements which are non-zero and non-constant functions of $t$. Those functions have the form $(t-\alpha )^k$. They are related to one of the $\alpha_i$'s. Let $\beta\in \{\alpha_i, \infty\mid 1\leq i \leq l\}$ denote the type of non-zero, non-constant entry which appears most. Now, the number $d_y(M')$ is the number of diagonal entries without the number of zeroes and without the number of appearances of a type $\beta$ entry, $2s_\beta$. 
		 
		 We will first argue that if $d > d_y(M')$, $I_d(M') \neq \left(1\right)$. We need to consider all $d\times d$ minors. Note that we may restrict our attention to principal minors, as all other minors of a diagonal matrix are zero. In essence, a $d\times d$ principal minor is precisely choosing $d$ diagonal entries. However, by the meaning of $d_y(M')$, choosing $d$ diagonal entries would imply always choosing at least one zero entry or one entry of type $\beta$. Hence, the greatest common divisor of all these minors must divide  $(t-\beta)$. In particular, $I_d(M') \neq \left(1\right)$.
		 
		 Now we end the proof by arguing that $I_{d_y(M')} = \left(1\right)$. Note that this ideal is generated by all principal minors of the given size. If we pick at least all the ones for our principal minor, and if we exclude any zeroes from it, then we must pick another 
		 $$d = \sum_{i=1}^l 2s_i - 2s_\beta$$
		 non-zero, non-constant entries. These entries can be divided into subsets according to the type of block it corresponds to. Here, $\beta$ represents the largest subset. This setting lends itself to using Lemma \ref{lem_distinct_tuples} above. Each tuple $T$ now corresponds to a minor $$M(T) := \prod_{i=1}^{\text{length of }T} q_{T(i)}(t),$$ where $q_{T(i)}$ is a polynomial of the type corresponding to the group of $T(i)$, e.g. $(t-\alpha )^r$. The power $r$ of the polynomial depends on the specific entry $T(i)$.
		 
		 Take two tuples $T_1$ and $T_2$ which agree on the first $d-1$ entries, say this part is a tuple $T'$ of length $d-1$, but have entries of a different group on the last position, say $(t-\alpha )^{r_1}$ and $(t- \alpha' )^{r_2}$ ($\alpha \neq \alpha'$). Then, we see that 
		 $$\gcd(M(T_1), M(T_2)) = M(T')\gcd((t-\alpha )^{r_1},(t- \alpha' )^{r_2}) = M(T'),$$
		 since the polynomials of different groups have no common zero.
		 
		 We can do this for every pair of tuples as above. We are left with $2^d/2$ tuples of length $d-1$ with the same defining properties as given in Lemma \ref{lem_distinct_tuples}. Hence, we can continue this procedure to show that the greatest common divisor all $2^d$ original tuples was in fact $1$. Therefore, $I_{d_y(M')} = \left(1\right)$. This ends the proof. 
	\end{proof}
	
	Now we must show that $y^{d} \in I_{d}(xA_1+yA_2)$ if $d = d_y(xA_1+yA_2)$. We first need a crucial lemma.
	\begin{lemma} \label{lem_resultant}
		Let $r_1,r_2 \in \mathbb{N}$, $\alpha \neq \alpha' \in \C$. The following statement about ideals in $\C[x,y]$ holds:
		$$\left((x-\alpha y)^{r_1}\right)\cdot H_{r_2-1}(x,y) + \left((x-\alpha' y)^{r_2}\right)\cdot H_{r_1-1}(x,y)  = H_{r_1+r_2-1}(x,y).$$
	\end{lemma}
	\begin{proof}
		Recall the notion of the resultant $\Res(q_1,q_2)$ of two polynomials $q_1(t) = \sum_{i=1}^n a_it^i$ and $q_2(t) = \sum_{i=1}^m b_it^i$ in $\C[t]$: this is the determinant of the matrix below.
		$$S(q_1,q_2) = \begin{pmatrix}a_n &&&b_m&& \\ a_{n-1}& a_n & & b_{m-1} & b_m &  \\ \vdots & \vdots & \ddots & \vdots & \vdots & \ddots \\ a_0 & a_1 & \dots & b_0 & b_1 & \dots \\ & a_0 & \dots &&b_0 & \dots \\&&\ddots&&&\ddots \end{pmatrix} $$
		The columns containing the coefficients of $q_1$ span $m$ columns, i.e. the degree of $q_2$, and the columns containing the coefficients of $q_2$ span $n$ columns, i.e. the degree of $q_1$.
		
		It is a known fact that $\Res(q_1,q_2) := \det(S(q_1,q_2)) \neq 0$ if and only if $q_1$ and $q_2$ have no common root in $\C$. If $q_1$ and $q_2$ have no common root, this fact also implies that $S(q_1,q_2)$ is a full rank matrix. Hence, for every vector $v\in \C^{m+n}$, we find a vector $(c, d) \in \C^{m+n}$ (with $c\in\C^m$ and $d\in \C^n$) such that
		$$v = S(q_1,q_2)\begin{pmatrix}c\\d\end{pmatrix}. $$
		This corresponds to the polynomial identity $v(t) = c(t)q_1(t) + d(t)q_2(t)$ with $v(t) = \sum_{j=1}^{n+m} v_j t^{n+m-j}$, $c(t) = \sum_{j=1}^{m} c_jt^{m-j}$ and $d(t) = \sum_{j=1}^n d_jt^{n-j}$. We have therefore made the observation that if $\Res(q_1,q_2) \neq 0$,
		\begin{equation} \label{eq_bezout}
		q_1(t) \cdot H_{m-1}(t) + q_2(t) \cdot H_{n-1}(t) = H_{n+m-1}(t),
		\end{equation}
		where $H_k(t)$ represents the vector space of polynomials of degree at most $k$.
		
		Set $t = x/y$. Now remark that $\Res((t-\alpha)^{r_1},(t-\alpha' )^{r_2}) \neq 0$, since they have no common root. If we homogenize Equation \eqref{eq_bezout}, we obtain the lemma's statement.
	\end{proof}
	\begin{prop} \label{prop_dy_in_ideal}
		In a skew matrix pencil $xA_1+ yA_2$, we have 
		$$y^{d_y(xA_1+yA_2)} \in I_{d_y(xA_1+yA_2)}(xA_1+yA_2).$$
	\end{prop}
	\begin{proof}
		It is a known fact that $I_k(M) = I_k(PMQ)$ if $P,Q \in \GL(m, \C)$, see for example \cite[Chapter 1, Theorem 2]{northcott1976finite}. Therefore, we may assume that $xA_1+yA_2$ has the form given in Theorem \ref{thm_pencils}, i.e.
		$$\left(\bigoplus_{i=1}^l \left( \bigoplus_{j=1}^{s_i} F(\alpha_i, k^{(i)}_j) \right) \right) \oplus \left( \bigoplus_{j=1}^{s_\infty} F(\infty, k^{(\infty)}_j) \right) \oplus \left( \bigoplus_{j=1}^{s_S} S(k^{(S)}_j) \right) \oplus 0.$$
		
		Define 
		$$M_1 =  \left( \bigoplus_{j=1}^{s_\infty} F(\infty, k^{(\infty)}_j) \right) \oplus \left( \bigoplus_{j=1}^{s_S} S(k^{(S)}_j) \right) \oplus 0$$
		and 
		$$M_2 = \bigoplus_{i=1}^l \left( \bigoplus_{j=1}^{s_i} F(\alpha_i, k^{(i)}_j) \right).$$
		Now define the following numbers:
		$$f_1 =  2\left(\sum_{j=1}^{s_\infty} k^{(\infty)}_j + \sum_{j=1}^{s_S} k^{(S)}_j \right) $$
		and
		$$f_2 = d_y(xA_1+yA_2)-f_1 =  2\left(\sum_{j=1}^{s_i} (k^{(i)}_j - 1) +  \sum_{i=1}^l s_i - s_\beta \right)$$
		with $s_\beta = \max\{s_i\mid 1\leq i \leq l\}$ (using Lemma \ref{lem_dphi_pencil}). 
		
		Due to the block structure of $xA_1+yA_2$, we see that
		$$I_{f_1}(M_1)\cdot I_{f_2}(M_2) \subset I_{d_y(xA_1+yA_2)}(xA_1+yA_2).$$
		We claim that $y^{f_1} \in I_{f_1}(M_1)$ and $I_{f_2}(M_2) = H_{f_2}(x,y)$. In particular, $y^{f_2} \in I_{f_2}(M_2)$ and 
		$$y^{f_1+f_2} = y^{d_y(xA_1+yA_2)} \in I_{d_y(xA_1+yA_2)}(xA_1+yA_2).$$
		
		Let us first show that $y^{f_1}$ lies in $I_{f_1}(M_1)$. For this, recall from Lemma \ref{lem_pencil_minors} that $\det(F(\infty,k)) = y^{2k}$ and there is a $2k\times 2k$ minor of $S(k)$ that equals $y^{2k}$. Taking these minors in the respective summands of $M_1$ illustrates that $y^{f_1}$ can be obtained as a minor of $M_1$.
		
		Let us now argue that $I_{f_2}(M_2) = H_{f_2}(x,y)$. Note, for starters, that every summand in $M_2$ consists of two blocks. This way, there are $n$ block in total with
		$$n = 2\sum_{i=1}^l s_i.$$
		These blocks can be divided in $l$ different possible types. This way, we can uniquely identify each block by a tuple $(i,j)$, where $i$ represents the type and $j$ represents the number of the block of this type, as we did before in Lemma \ref{lem_distinct_tuples}. Now define $d$ to be the number $d = n - 2s_\beta$, and apply Lemma \ref{lem_distinct_tuples} to obtain $2^d$ distinct tuples of length $d$. If $T$ is a tuple, we associate an ideal $I(T)$ with it as follows:
		$$I(T) =  \left(\prod_{(i,j) \in T} \det(i,j)\right)\cdot \left(\prod_{(i,j)\notin T}I_{k(i,j)-1}(i,j)\right).$$
		Here, $k(i,j)$ represents the size of the block named $(i,j)$, and $I_{k(i,j)-1}(i,j)$ hence represents the ideal of $(k(i,j)-1)\times(k(i,j)-1)$ minors within the block $(i,j)$. The notation $\det(i,j)$ is used as short hand for $I_{k(i,j)}(i,j)$. In summary, $I(T)$ should be interpreted as the ideal obtained by `fully taking the blocks in $T$' and `only taking the free space in the remaining blocks'. Indeed, $I_{k(i,j)-1}(i,j) = H_{k(i,j)-1}(x,y)$ by Lemma \ref{lem_pencil_minors}. This way, the given ideal is actually generated by minors of $M_2$ with size $f_2$. We conclude that if $V$ is the set of tuples $T$ then 
		$$\sum_{T\in V} I(T) \subset I_{f_2}(M_2).$$
		
		We claim that this sum equals $H_{f_2}(x,y)$.
		
		The $2^d$ tuples can be divided in pairs $\{T_a,T_b\}$ such that $T_a(k) = T_b(k)$ for all $1\leq k \leq d-1$. We know that $T_a(d)$ and $T_b(d)$ are of a different type. Note that in particular, $T_a(d)$ does not appear in $T_b$ and vice versa. Thus, we see that
		\begin{multline*}
			I(T_a) + I(T_b)  = \left(\prod_{k=1}^{d-1} \det T_a(k)\right) \cdot \left( \det(T_a(d))\cdot H_{k(T_b(d))-1}(x,y) + \det(T_b(d))\cdot H_{k(T_a(d))-1}(x,y) \right) \\ \cdot \left(\prod_{(i,j)\notin T_a\cup T_b}H_{k(i,j)-1}(x,y)\right).
		\end{multline*}
		Recall that $\det(T_a(d))$ has the form $\left((x-\alpha y)^r\right)$. By Lemma \ref{lem_resultant}, we know that
		\begin{equation*} 
			\det(T_a(d))\cdot H_{k(T_b(d))-1}(x,y) + \det(T_b(d))\cdot H_{k(T_a(d))-1}(x,y) = H_{k(T_b(d))+k(T_a(d))-1}(x,y).
		\end{equation*} 
		Therefore, we obtain
		\begin{equation}\label{eq_sum_of_ideals}
			I(T_a) + I(T_b)  = \left(\prod_{k=1}^{d-1} \det T_a(k)\right) \cdot H(x,y),
		\end{equation}
	where $H(x,y)$ contains homogeneous polynomials of degree 
	$$k(T_b(d))+k(T_a(d))-1 + \sum_{(i,j) \notin T_a\cup T_b} (k(i,j)-1).$$
	
	Now for every pair $\{T_a,T_b\}$ we define $T'$ to be the tuple of length $d-1$ consisting of the first $(d-1)$ entries of $T_a$ (or $T_b$). We know that the $2^d/2$ tuples $T'$ satisfy the same properties given in Lemma \ref{lem_distinct_tuples}. Define $I(T')$ as in Equation \eqref{eq_sum_of_ideals}. Note that the degree of $H(x,y)$ in this expression is larger than $k(i,j)-1$ for all $(i,j) \notin T'$. Hence, if $\{T',T''\}$ is a pair with the first $(d-2)$ entries in common, we can use the same trick:
	\begin{equation*}
		\begin{split}
			I(T') + I(T'') & =  \left(\prod_{k=1}^{d-2} \det T'(k)\right) \cdot \left( \det(T'(d-1))\cdot H(x,y) + \det(T''(d-1))\cdot H(x,y)\right) \\ &= \left(\prod_{k=1}^{d-2} \det T'(k)\right) \cdot H(x,y),
		\end{split}
	\end{equation*}
	where every $H(x,y)$ has polynomials of the appropriate degree.	
	
	Continuing this procedure, we must conclude that 
	$$H(x,y) = \sum_{T\in V} I(T) \subset I_{f_2}(M_2),$$
	and hence $I_{f_2}(M_2) = H_{f_2}(x,y)$.
	\end{proof}

\bibliographystyle{plain}
\bibliography{Dere_Matthys_two_step_nilp}

\begin{thebibliography}{10}

\bibitem{ax1967solving}
James Ax.
\newblock Solving diophantine problems modulo every prime.
\newblock {\em Ann. of Math. (2)}, 85:161--183, 1967.

\bibitem{bou2011asymptotic}
K.~Bou-Rabee and D.~B. McReynolds.
\newblock Asymptotic growth and least common multiples in groups.
\newblock {\em Bull. Lond. Math. Soc.}, 43(6):1059--1068, 2011.

\bibitem{bou2010quantifying}
Khalid Bou-Rabee.
\newblock Quantifying residual finiteness.
\newblock {\em J. Algebra}, 323(3):729--737, 2010.

\bibitem{bou2019residual}
Khalid Bou-Rabee, Junjie Chen, and Anastasiia Timashova.
\newblock Residual finiteness growths of lamplighter groups.
\newblock {\em arXiv preprint arXiv:1909.03535}, 2019.

\bibitem{bou2015residual}
Khalid Bou-Rabee, Mark~F. Hagen, and Priyam Patel.
\newblock Residual finiteness growths of virtually special groups.
\newblock {\em Math. Z.}, 279(1-2):297--310, 2015.

\bibitem{bou2012quantifying}
Khalid Bou-Rabee and Tasho Kaletha.
\newblock Quantifying residual finiteness of arithmetic groups.
\newblock {\em Compos. Math.}, 148(3):907--920, 2012.

\bibitem{bradford2019short}
Henry Bradford and Andreas Thom.
\newblock Short laws for finite groups and residual finiteness growth.
\newblock {\em Trans. Amer. Math. Soc.}, 371(9):6447--6462, 2019.

\bibitem{MR3793294}
Yves~de Cornulier.
\newblock On the quasi-isometric classification of locally compact groups.
\newblock In {\em New directions in locally compact groups}, volume 447 of {\em
  London Math. Soc. Lecture Note Ser.}, pages 275--342. Cambridge Univ. Press,
  Cambridge, 2018.

\bibitem{corwin1990representations}
Lawrence~J. Corwin and Frederick~P. Greenleaf.
\newblock {\em Representations of nilpotent {L}ie groups and their
  applications. {P}art {I}}, volume~18 of {\em Cambridge Studies in Advanced
  Mathematics}.
\newblock Cambridge University Press, Cambridge, 1990.
\newblock Basic theory and examples.

\bibitem{dd14-1}
Karel Dekimpe and Jonas Der\'e.
\newblock Existence of {A}nosov diffeomorphisms on infra-nilmanifolds modeled
  on free nilpotent {L}ie groups.
\newblock {\em Topol. Methods Nonlinear Anal.}, 46(1):165--189, 2015.

\bibitem{dere2023residual}
Jonas Der\'e and Joren Matthys.
\newblock Residual finiteness growth in virtually abelian groups.
\newblock {\em J. Algebra}, 657:482--513, 2024.

\bibitem{survey2022}
Jonas Deré, Michal Ferov, and Mark Pengitore.
\newblock Survey on effective separability.
\newblock {\em Accepted for publication in "Geometric methods in group theory:
  papers dedicated to Ruth Charney" as part of the series Séminares et
  Congrès by the SMF}, 2022.

\bibitem{franz2017quantifying}
Daniel Franz.
\newblock Quantifying residual finiteness of linear groups.
\newblock {\em J. Algebra}, 480:22--58, 2017.

\bibitem{gauger1973classification}
Michael~A. Gauger.
\newblock On the classification of metabelian {L}ie algebras.
\newblock {\em Trans. Amer. Math. Soc.}, 179:293--329, 1973.

\bibitem{grunewald1982nilpotent}
Fritz~J. Grunewald, Daniel Segal, and Leon~S. Sterling.
\newblock Nilpotent groups of {H}irsch length six.
\newblock {\em Math. Z.}, 179(2):219--235, 1982.

\bibitem{haslinger2001families}
Gareth~Jon Haslinger.
\newblock {\em Families of skew-symmetric matrices}.
\newblock PhD thesis, University of Liverpool, 2001.

\bibitem{hils2019first}
Martin Hils and Fran\c{c}ois Loeser.
\newblock {\em A first journey through logic}, volume~89 of {\em Student
  Mathematical Library}.
\newblock American Mathematical Society, Providence, RI, 2019.

\bibitem{MR1683874}
Fernando Levstein and Alejandro Tiraboschi.
\newblock Classes of {$2$}-step nilpotent {L}ie algebras.
\newblock {\em Comm. Algebra}, 27(5):2425--2440, 1999.

\bibitem{northcott1976finite}
D.~G. Northcott.
\newblock {\em Finite free resolutions}, volume No. 71 of {\em Cambridge Tracts
  in Mathematics}.
\newblock Cambridge University Press, Cambridge-New York-Melbourne, 1976.

\bibitem{pengitore2018corrigendum}
Mark Pengitore.
\newblock Corrigendum to "{E}ffective separability of finitely generated
  nilpotent groups", {\it {n}ew {y}ork {j}. {m}ath.} 24 (2018), 83--145.
  [{MR}3761940].
\newblock {\em New York J. Math.}, 24:897--901, 2018.

\bibitem{sega83-1}
Daniel Segal.
\newblock {\em Polycyclic groups}, volume~82 of {\em Cambridge Tracts in
  Mathematics}.
\newblock Cambridge University Press, Cambridge, 1983.

\bibitem{serre2016lectures}
Jean-Pierre Serre.
\newblock {\em Lectures on {$N_X (p)$}}, volume~11 of {\em Chapman \& Hall/CRC
  Research Notes in Mathematics}.
\newblock CRC Press, Boca Raton, FL, 2012.

\bibitem{tholozan2022residually}
Nicolas Tholozan and Konstantinos Tsouvalas.
\newblock Residually finite non linear hyperbolic groups.
\newblock {\em arXiv preprint arXiv:2207.14356}, 2022.

\bibitem{van1991remark}
Lou van~den Dries.
\newblock A remark on {A}x's theorem on solvability modulo primes.
\newblock {\em Math. Z.}, 208(1):65--70, 1991.

\bibitem{Winter1974Fields}
David Winter.
\newblock {\em The structure of fields}, volume No. 16.
\newblock Springer-Verlag, New York-Heidelberg, 1974.

\end{thebibliography}
\end{document}